\documentclass[a4paper,10pt]{amsart}
\usepackage[utf8x]{inputenc}

\usepackage{amsmath}
\usepackage{amssymb}
\usepackage{amsthm}

\providecommand{\cO}{\mathcal{O}}
\providecommand{\Gm}{\mathbb{G}_m}
\providecommand{\calL}{\mathcal{L}}
\providecommand{\calH}{\mathcal{H}}

\newcommand{\Hyp}{\mathfrak{H}^3}

\newcommand{\bbR}{\mathbb{R}}
\newcommand{\bbQ}{\mathbb{Q}}
\newcommand{\bbZ}{\mathbb{Z}}
\newcommand{\bbH}{\mathbb{H}}
\newcommand{\bbC}{\mathbb{C}}
\newcommand{\bbA}{\mathbb{A}}
\newcommand{\bbN}{\mathbb{N}}

\newcommand{\Ia}{\mathfrak{a}}
\newcommand{\Ip}{\mathfrak{p}}
\newcommand{\IP}{\mathfrak{P}}
\newcommand{\IQ}{\mathfrak{Q}}

\newcommand{\co}{\mathfrak{o}}
\newcommand{\bsl}[2]{\ensuremath{#1}\backslash\ensuremath{#2}}

\newcommand{\up}[1]{\,^{#1}}

\DeclareMathOperator{\Gal}{Gal}
\DeclareMathOperator{\id}{Id}
\DeclareMathOperator{\GL}{GL}
\DeclareMathOperator{\SL}{SL}
\DeclareMathOperator{\PGL}{PGL}
\DeclareMathOperator{\PSL}{PSL}
\DeclareMathOperator{\SU}{SU}
\DeclareMathOperator{\nrd}{nrd}
\DeclareMathOperator{\Nr}{N}
\DeclareMathOperator{\tr}{Tr}
\DeclareMathOperator{\im}{Im}
\DeclareMathOperator{\inn}{int}
\DeclareMathOperator{\Aut}{Aut}
\DeclareMathOperator{\Ram}{Ram}
\DeclareMathOperator{\vol}{vol}
\DeclareMathOperator{\disc}{d}
\DeclareMathOperator{\N}{N}
\DeclareMathOperator{\Res}{Res}
\DeclareMathOperator{\rk}{rk}
\DeclareMathOperator{\Br}{Br}

\theoremstyle{plain}
\newtheorem{theorem}{Theorem}[section]
\newtheorem{corollary}[theorem]{Corollary}
\newtheorem{lemma}[theorem]{Lemma}

\newtheorem*{maintheorem}{Main Theorem}
\newtheorem*{theorem*}{Theorem}

\theoremstyle{definition}
\newtheorem{definition}[theorem]{Definition}
\newtheorem{remark}[theorem]{Remark}

\title[On the growth of the first Betti number]{On the growth of the first Betti number of arithmetic hyperbolic 3-manifolds}
\author[S. Kionke]{Steffen Kionke}
\address{Faculty of Mathematics,
         University of Vienna, 
         Nordbergstrasse 15, 
         A-1090 Vienna, Austria.}
\email{steffen.kionke@univie.ac.at}
\author[J. Schwermer]{Joachim Schwermer}
\address{ Faculty of Mathematics,
         University of Vienna, 
         Nordbergstrasse 15, 
         A-1090 Vienna, Austria.\newline
         \indent Erwin Schr\"odinger International Institute for Mathematical Physics,
         Boltzmanngasse 9, A-1090 Vienna, Austria.}
\email{Joachim.Schwermer@univie.ac.at}

\thanks{The authors were supported by FWF Austrian Science Fund, grant number P 21090-N13}
\date{\today}

\begin{document}
\begin{abstract}
We calculate the Lefschetz number of a Galois automorphism in the cohomology of certain arithmetic congruence groups
arising from orders in quaternion algebras over number fields.
As an application we give a lower bound for the first Betti number of a class of arithmetically defined hyperbolic
$3$-manifolds and we deduce the following theorem: 
Given an arithmetically defined cocompact subgroup $\Gamma \subset \SL_2(\bbC)$,
provided the underlying quaternion algebra meets some conditions, there is
a decreasing sequence $\{\Gamma_i\}_i$ of finite index subgroups of $\Gamma$ such that the first 
Betti number satisfies \[ b_1(\Gamma_i) \gg [\Gamma:\Gamma_i]^{1/2} \]  as $i$ goes to infinity.
\end{abstract}

\maketitle

\section{Introduction}

\subsection{The first Betti number of hyperbolic $3$-manifolds}
  Every orientable hyperbolic 3-manifold is isometric to the quotient $\Hyp/\Gamma$ of hyperbolic
  $3$-space $\Hyp$ by a discrete torsion-free subgroup $\Gamma$ of the group of 
  orientation-preserving isometries of $\Hyp$. The latter group is isomorphic to the
  connected group $\PGL_2(\bbC)$, the real Lie group $\SL_2(\bbC)$ modulo its centre.
  Generally, a discrete subgroup of $\PGL_2(\bbC)$ is called a Kleinian group. 

  Within Thurston's geometrization program and its subsequent treatment by Perelman the class of hyperbolic 3-manifolds
  plays a fundamental role but is still not well understood. One of the open problems is the fundamental
  conjecture in 3-manifold theory, stated by Waldhausen \cite{Waldhausen1968} in 1968, which says:
  given an irreducible 3-manifold $M$ with infinite fundamental group there exists a finite cover $M'$ of $M$
  which is Haken, that is, it is irreducible and contains an embedded incompressible surface.
  This so-called virtual Haken conjecture is the source for the even stronger virtual positive Betti number
  conjecture. Within the class of hyperbolic $3$-manifolds it states that,
  given a hyperbolic $3$-manifold $M$ there exists a finite cover $M'$ with non-vanishing first Betti number $b_1(M')$.  

  These two conjectures concern finite covering spaces of 3-manifolds. 
  Thus, one is naturally led to ask how various algebraic or geometric invariants attached to hyperbolic $3$-manifolds
  behave in finite-sheeted covers (see e.g. \cite{Lackenby2010}). 
  Our object of concern will be the first Betti number in the case of arithmetically defined hyperbolic $3$-manifolds. 
  Among hyperbolic $3$-manifolds, the ones originating with arithmetically defined Kleinian groups form a class of special interest.
  Due to the underlying connections with number theory this class seems to be in many ways more tractable.
  Indeed, in many cases geometric techniques or results in the theory of automorphic forms 
  make it possible to construct non-vanishing (co)-homology classes on these hyperbolic manifolds up to a finite cover
  (see, e.g., \cite{Schwermer2010}, chap. I, for a survey of various results  or,
  \cite{LabesseSchwermer1986}, \cite{Lubotzky1996}, \cite{Rajan2004}).

  Investigating the first Betti number, it is quite natural to consider its growth rate in a nested  sequence
  $\{\Gamma_i \}_{i \in \bbN}$ of finite index (normal) subgroups $\Gamma_i \subset \Gamma$ (whose intersection is the identity)
  for a given arithmetically defined Kleinian group $\Gamma$. One defines the first Betti number gradient which is the limit
  of the ratio of the first Betti number $b_1(\Gamma_i)$ by the index $[\Gamma : \Gamma_i ]$. 
  This is a special case of a general concept:
  Let  $\Gamma$  be a lattice in a semi-simple real Lie group $G$. If $\{\Gamma_i \}_{i \in \bbN}$ is a nested sequence of
  finite index normal subgroups $\Gamma_i \subset \Gamma$ (whose intersection is the identity) 
  one can form the quotients 
  \begin{equation*}
    \beta_j(\Gamma_i) = \frac {\dim  H_j(\Gamma_{i}, \bbC)} {[\Gamma : \Gamma_i]}.
  \end{equation*}
  It is known by a result of L\"uck \cite{Lueck1994} that the $\beta_j(\Gamma_i)$ converge to the j-th $L^2$-Betti number
  of $\Gamma$, that is, the limit $\lim_{i} \beta_j(\Gamma_i)$ exists for each $j$.
  The limit is non-zero if and only if the rank $\rk_{\bbC}G$ of $G$ equals the rank $\rk_{\bbC}K$ of a maximal compact
  subgroup $K \subset G$ and $j = \frac{1}{2} \dim(G/K)$. There are several works, 
  notably by de George and Wallach \cite{deGeorgWallach1978}, Savin \cite{Savin1989}, and Rohlfs and Speh \cite{RohlfsSpeh1987-2} among others, in which one finds precise 
  results pertaining to the actual value of this limit in the case where $\delta(G): =  \rk_{\bbC}G  - \rk_{\bbC}K  = 0$.

  However, in the situation of arithmetically defined hyperbolic $3$-manifolds, that is,
  $G$ is the group $\PGL_2(\bbC)$ one has  $\delta(G): =  \rk_{\bbC}\PGL_2(\bbC)  - \rk_{\bbC}K  = 1$, thus, 
  \begin{equation*}
   \lim_{i} \beta_j(\Gamma_i) = 0.
  \end{equation*}
  In particular, this assertion is valid for $j = 1$. As a consequence, the sequence of first
  Betti numbers $b_1(\Gamma_i)$ grows sub-linearly as a function of the index $[\Gamma : \Gamma_i]$ whenever $\{\Gamma_i\}_{i\in\bbN}$
  is a decreasing sequence of finite index normal\footnote{The
  conclusion still holds if, for instance, the $\Gamma_i$ are not normal in $\Gamma$ but $\Gamma_i$ is normal in $\Gamma_1$ for all $i$.}
   subgroups in an arithmetically defined group $\Gamma \subset \PGL_2(\bbC)$.
  Recently there has been some progress on improved upper bounds for the growth of 
  Betti numbers, e.g. in \cite{CalegariEmerton2009} and \cite{ClairWhyte2003}.
  Our objective is to deduce \emph{lower bounds} for the growth of the first Betti number.
   
  Our main result concerns a specific class (see below) of compact arithmetically defined hyperbolic $3$-manifolds which
  originate with orders in suitable division quaternion algebras $D$ defined over some number field $E$.
  Given an arithmetic subgroup in the algebraic group $\SL_1(D)$ we show that there are a positive real number $\kappa$
  and a nested sequence $\{\Gamma_i \}_{i \in \bbN}$ of finite index subgroups $\Gamma_i \subset \Gamma$ 
  (whose intersection is the identity) such that the first Betti number of the compact hyperbolic $3$-manifold $\Hyp/\Gamma_i$
  corresponding to $\Gamma_i$ satisfies the inequality
  \begin{equation*}
     b_1(\Gamma_i) \geq \kappa [\Gamma : \Gamma_i]^{1/2}
  \end{equation*}
  for all indices $i \in \bbN$. 
  One obtains a similar result in the case of Bianchi groups, that is, the corresponding manifold is non-compact. In this 
  case one can construct nested sequences such that the first Betti number grows at least as fast as $[\Gamma:\Gamma_i]^{2/3}$ up to a factor.

  In the following subsections, we precisely describe the class of hyperbolic $3$-manifolds in question and
  give an exact formulation of the results obtained.

\subsection{Arithmetically defined hyperbolic $3$-manifolds} 
  For the sake of convenience we begin with the notion of an arithmetically defined Kleinian group.
  A discrete subgroup $\Gamma$ of $\PGL_2(\bbC)$ is said to be arithmetically defined if there exists an algebraic number field $E/\bbQ$
  with exactly one complex place $w$, an arbitrary (but possibly empty) set $T$ of real places, an $E$-form $G$ of 
  the algebraic $E$-group $\PGL_{2}/E$ such that $G(E_v)$ is a compact group for all $v \in T$ and an isomorphism
  $\PGL_2(\bbC) \longrightarrow G(E_w)$, which maps $\Gamma$ onto an arithmetic subgroup of the
  group $G(E)$ of $E$-points naturally embedded into $G(E_w)$.
  The corresponding hyperbolic $3$-manifolds $\Hyp/\Gamma$ fall naturally into two classes, according to whether $\Hyp/\Gamma$
  is compact or not\footnote{However, this quotient always has finite volume.}. 

  In the latter case, $E/\bbQ$ is an imaginary quadratic extension, the group $G$ is the split form $\PGL_2/E$ itself,
  and the set $T$ is the empty set. The groups in question are the subgroups of $\PGL_2(E)$ which are commensurable 
  with the group $\PGL_2(\cO_E)$ where $\cO_E$ denotes the ring of integers in $E$.
  These are the groups already considered by L. Bianchi in 1892. 

  In the former case, given the algebraic number field $E$ with exactly one complex place,
  we consider an $E$-form $G$ of $\PGL_2/E$ originating 
  from a quaternion division algebra $D$ over $E$ which ramifies at least at all
  real places  $v \in T$. Given an order $\Lambda$ in $D$ any torsion-free subgroup $\Gamma$ in the group $\SL_1(D)$ of elements
  of reduced norm one in $D$, which is commensurable 
  with $\SL_1(\Lambda)$ gives rise to a compact $3$-manifold $\Hyp/\Gamma$.

  A torsion-free discrete subgroup in $\SL_2(\bbC)$ projects isomorphically to a torsion-free discrete group in $\PGL_2(\bbC)$. 
  Therefore we shall only consider arithmetically defined groups in inner forms of $\SL_2/E$.

\subsection{The main result}
  We are mainly concerned with arithmetically defined hyperbolic $3$-manifolds and corresponding Kleinian groups which originate
  with orders in division quaternion algebras defined over some algebraic number field $E$.
  Before we state our main result, we give a description of the class of quaternion algebras to which the main theorem applies.
  We suppose that the field $E$, has exactly one complex place and an arbitrary (possibly empty) set $T$ of real places.
  Moreover, we assume that $E$ contains a subfield $F$ such that the degree of the extension $E/F$ is two. 
  Then $F$ is a totally real extension field of $\bbQ$. Let $\sigma$ denote the non-trivial element
  in the cyclic Galois group $\Gal(E/F)$ of the extension $E/F$.

  Let $D$ denote a quaternion division algebra over $E$ such that the finite set of places ramified in $D$ contains the set $T$ of real
  places of $E$.
  As a quaternion division algebra, $D$ is isomorphic to its opposite algebra, and the class
  of $D$ in the Brauer group $\Br(E)$ of $E$ is of order two.
  Thus, the norm $\N_{E/F}(D)$, a central simple algebra of degree $4$ over $F$,
  has order $1$ or $2$ viewed as an element in the Brauer group $\Br(F)$.
  Recall that the unit element in the Brauer group is the class of $F$ or, equivalently, the class of all matrix algebras over $F$. 

  We distinguish the two cases 
  \begin{itemize}
  \item[(I)]  The class $[\N_{E/F} (D)]$ has order $1$ in $\Br(F)$,
  \item[(II)] The class $[\N_{E/F} (D)]$ has order $2$ in $\Br(F)$.
  \end{itemize}

  In case (I), the $F$-algebra $\N_{E/F} (D)$ is
  isomorphic to the matrix algebra $M_4(F)$, that is, it splits. By a result of Albert und
  Riehm (cf. \cite[(3.1)]{BOInv}), $\N_{E/F} (D)$ splits if and only if there is an involution of the second
  kind on $D$ which fixes $F$ elementwise. Let $\tau$ denote this involution of the second kind.
  By definition of this notion, the restriction of $\tau$ to the center of $D$ is of order $2$, hence
  $\tau_{\vert E}$ coincides with $\sigma$. As Albert has proved (cf. \cite[(2.22)]{BOInv}), an involution of the second
  kind on a quaternion algebra has a particular type. There exists a unique quaternion
  $F$-subalgebra $D_0 \subset D$ such that $D = D_{0} \otimes_{F} E$ and $\tau$ is of the form $\tau  = \gamma_{0} \otimes \sigma$ where
  $\gamma_0$ is the canonical involution (also called quaternion conjugation) on $D_0$. The algebra
  $D_0$ is uniquely determined by these conditions. 

  We will consider the involution $\id_{D_0} \otimes \sigma$ on $D = D_{0} \otimes_{F} E$ induced by the 
  non-trivial Galois automorphism $\sigma$ of the extension $E/F$. 
  For the sake of simplicity it will  be denoted by the same letter $\sigma$.

 In case (II),  the $F$-algebra $N_{E/F} (D)$ of degree $4$ is (up to isomorphism) of the form $M_2(Q)$,
 where $Q$ is a quaternion division algebra over $F$.

\begin{theorem*} 
  Let $F$ be a totally real algebraic number field, and let $E$ be a quadratic extension field of $F$ so that $E$ has
  exactly one complex place. Let $\Gamma$ be an arithmetic subgroup in the algebraic group $\SL_1(D)$ where $D$
  is a quaternion division algebra over $E$ which belongs to case (I). Then there are a positive number $\kappa > 0$
  and a nested sequence $\{\Gamma_i \}_{i \in \bbN}$ of torsion-free, finite index subgroups $\Gamma_i \subset \Gamma$ 
  (whose intersection is the identity) such that the first Betti number of the compact hyperbolic
   $3$-manifold $\Hyp/\Gamma_i$ corresponding to $\Gamma_i$ satisfies the inequality
   \begin{equation*}
        b_1(\Gamma_i) \geq \kappa [\Gamma : \Gamma_i]^{1/2}
   \end{equation*}
   for all indices $i \in \bbN$.
   Further, $\Gamma_i$ is normal in $\Gamma_1$ for all $i\in \bbN$.
\end{theorem*}

  The proof of this result relies on the following methodological approach:
  The non-trivial Galois automorphism $\sigma$ of the extension $E/F$ induces an orientation-reversing
  involution on the hyperbolic $3$-manifold $\Hyp/\Gamma$, whenever $\Gamma$ is $\sigma$-stable. 
  In the case the extension $E/F$ is unramified over $2$ one can determine the Lefschetz number
  $\calL(\sigma, \Gamma)$ of the induced homomorphism in the cohomology of $\Hyp/\Gamma$ where 
  $\Gamma$ is a suitable congruence subgroup in $\SL_1(D)$. 
  In the general case, one gets the analogous value as a lower bound for $\calL(\sigma, \Gamma)$.
  This bound is given up to sign and some power of two as 
   \begin{equation*}
       \pi^{-2d} \zeta_{F}(2)  |\disc_{F}|^{3/2} \Delta (D_0) \times  [K_0 : K_0(\mathfrak{a})],
   \end{equation*}
  where $\zeta_{F}(2)$ denotes the value of the zeta-function of $F$ at $2$,  $|\disc_{F}|$ denotes
  the absolute value of the discriminant of $F$,  $[K_0 : K_0(\mathfrak{a})]$ denotes a global index attached
  to the congruence subgroup of level $\Ia \subseteq \cO_F$,
  and $ \Delta (D_0) =  \prod_{\Ip_0 \in \Ram_{f}(D_0)} (\N_{F/\bbQ}(\Ip_0) - 1)$ depends
  on the set of finite places of $F$ in which the quaternion division algebra $D_0$ ramifies.
  In turn, this bound can be used to give a lower bound for the first Betti number of the hyperbolic $3$-manifold in 
  question  (see Theorem \ref{thmH1Estimate} and Corollary \ref{corEstimateH1Index}). This result implies that the first Betti number becomes arbitrarily large 
  when we vary over the congruence condition since the term $[K_0 : K_0(\Ia)]$ is unbounded.

\subsection{Outline}  
  We outline the content of the paper:
  In Section \ref{sectQuatAlg}, we give some background material pertaining
  to quaternion algebras $D$ defined over number fields and the corresponding algebraic groups $\SL_1(D)$ of reduced norm one elements.
  In this and the subsequent section we work in the general case of an arbitrary quadratic extension $E/F$ of a totally real number field $F$.
  In Section \ref{sectLefschetz}, we first outline the approach on which our result is based.
  The Lefschetz number of the orientation-reversing automorphism $\sigma$ of the manifold $\Hyp/\Gamma$ 
  is equal to the Euler characteristic of the space $(\Hyp/\Gamma)^{\sigma}$ of points in $\Hyp/\Gamma$ fixed under $\sigma$.
  The latter space and its connected components, interpreted in the language of adele groups, 
  can be described in terms of non-abelian Galois cohomology, following a general approach of Rohlfs (cf. \cite{Rohlfs1990}).
  The Euler characteristics in question can be calculated via an Euler-Poincar\'e measure. We compare this measure with
  the Tamagawa measure, which allows us to determine the Euler characteristic as an infinite product of local factors indexed
  by the finite places of the underlying field. Theorem \ref{thmLefschetzNumber} gives then the final result for the Lefschetz 
  number attached to $\sigma$ and a congruence subgroup in $\SL_1(D)$. 
  Section \ref{sectEstimates} contains some estimates for ratios of subgroup indices which occur by passing
  from congruence subgroups over $F$ to such over $E$.
  Finally, in Section \ref{sectHyperbolicApplications}, we apply the previous result in the case of arithmetically
  defined hyperbolic $3$-manifolds and we obtain the main result as indicated above.
  Some comments on how to obtain a similar result for Bianchi groups can be found in Section \ref{sectBianchi}.
  Moreover, there is an appendix in Section \ref{secAppendixLocal} where we stored some auxiliary, purely local results pertaining 
  to non-abelian Galois cohomology.

\specialsection*{Notation}

  We write $\bbZ$, $\bbQ$, $\bbR$, and $\bbC$ for the ring of integers, the field of rational numbers, 
  the field of real numbers,  
  and the field of complex numbers respectively.

  \subsection*{} (1) Let $K$ be an algebraic number field, i.e. a finite extension of the field $\bbQ$.
  The ring of algebraic integers of $K$ is denoted by $\cO_K$.
  Let $V(K)$ denote the set of places of $K$. The subsets of archimedean (resp. non-archimedean) places
  will be denoted $V_\infty(K)$ (resp. $V_f(K)$). Given a place $v \in V(K)$ the 
  completion of $K$ with respect to $v$ is denoted $K_v$. 
  For a finite place $v \in V_f(K)$ we write $\cO_{K,v}$ for the valuation ring in $K_{v}$.
  The symbol $\bbA_K$ denotes the ring of adeles of $K$. We use the notation $\bbA_{K,f}$ 
  for the ring finite adeles of $K$.

  \subsection*{} (2) All group schemes considered are affine and of finite type. Let $k$ be a commutative ring and
     $H$ a group scheme over $k$. Given any commutative $k$-algebra $R$, we write $H(R)$ for the group of $R$-rational
     points of $H$.

   \subsection*{} (3) We freely use the language of non-abelian Galois cohomology as defined by Serre \cite[I.§5]{Serre1964}.
   Whenever $H$ is a group on which the two element group acts by an automorphism called $\sigma$,
   we will denote the action by upper left exponents, i.e. $\up{\sigma}h$.
   Recall that a 1-cocycle for $\sigma$ with values in $H$ is an element $h \in H$ such that $h\up{\sigma}h = 1$.
   The set of such 1-cocycles will be denoted by $Z^1(\sigma, H)$. Two cocycles $h,g \in H$ are said to be equivalent,
   if there is some $b \in H$ such that $b^{-1}h\up{\sigma}b = g$.
   The first non-abelian cohomology set $H^1(\sigma,H)$ of $\sigma$ with values in $H$ is the set of equivalence classes for this relation.
   In general $H^1(\sigma, H)$ is not a group, but it is a pointed set where the class of the trivial cocycle $1_H$ is the distinguished point.

\section{Quaternion algebras and associated algebraic groups}\label{sectQuatAlg}

  \subsection{} 
   Throughout the article $F$ denotes a totally real algebraic number field and $E/F$ a quadratic extension of $F$.
   In section \ref{sectLefschetz} we impose no assumptions on $E$. 
   However, in section \ref{sectHyperbolicApplications} the field $E$ will be assumed to have precisely one complex place.
   We tried to consistently denote ideals in $\cO_F$ by fraktur letters indexed by zero (e.g. $\Ia_0$) whereas ideals in
   $\cO_E$ will have no subscript.
   Moreover, let $D_0$ be a quaternion algebra defined over $F$. Taking the tensor product with $E$, we obtain the
   quaternion algebra $D := D_0 \otimes_F E$ over $E$.
   We fix once and for all a maximal $\cO_F$-order $\Lambda_0$ in $D_0$. Further, we obtain an $\cO_E$-order
   $\Lambda := \Lambda_0 \otimes_{\cO_F} \cO_E$ in $D$. Surprisingly, this is in general not a maximal order in $D$ and it is valuable
   to keep that in mind.
 
  The finite set of places in $V(F)$ ramified in $D_0$ will be denoted by $\Ram(D_0)$. 
  As before, we write $\Ram_f(D_0)$ (resp. $\Ram_\infty(D_0)$) for the finite (resp. infinite) places in $\Ram(D_0)$.
  We write $r = |\Ram_\infty(D_0)|$ for the number of real ramified places and $s = [F:\bbQ]-r$ for the number of split places.

  \subsection{}\label{subsecGroupSchemes} With the given data several group schemes are associated. We write $\GL_1(\Lambda_0)$ for
   the $\cO_F$ group scheme of units associated with $\Lambda_0$. This means for a commutative $\cO_F$-algebra 
   $R$ we have $\GL_1(\Lambda_0)(R) := (\Lambda_0 \otimes_{\cO_F} R)^{\times}$.
   The reduced norm gives a morphism of group schemes
    \begin{equation*}
        \nrd: \GL_1(\Lambda_0) \to \Gm
    \end{equation*}
     into the multiplicative group defined over $\cO_F$. The kernel of the reduced norm 
     is a smooth $\cO_F$ group scheme denoted by $G_0 := \SL_1(\Lambda_0)$.
     Note that, taking the base change to $\cO_E$, we get the group $SL_1(\Lambda) = G_0 \times_{\cO_F} \cO_E$.
     Finally, we apply the (Weil) restriction of scalars and obtain another $\cO_F$ group scheme 
     \begin{equation*}
            G := \Res_{\cO_E / \cO_F} ( G_0 \times_{\cO_F} \cO_E ).
     \end{equation*}
     Moreover, the scheme $G$ is smooth over $\cO_F$.

   \subsection{} Let $\sigma$ denote the generator of the Galois group $\Gal(E/F)$. The Galois automorphism $\sigma$ induces
    an $F$-algebra automorphism $\id_{D_0} \otimes \sigma : D \to D$.
    For simplicity we will denote this morphism again by $\sigma$. Moreover, $\sigma$ induces an automorphism of order two on $G$. 
    We will still write $\sigma$ for this automorphism. One should notice that $\sigma: G \to G$
    is defined over $\cO_F$.
    Note that the group $G^{\sigma}\times_{\cO_F} F$  of $\sigma$-fixed points (over $F$) is canonically isomorphic to $G_0\times_{\cO_F}F$.
    In general the groups $G^\sigma$ and $G_0$ are not isomorphic over $\cO_F$.

   \subsection{} Define the real Lie group
    \begin{equation*}
        G_{\infty} := \prod_{v \in V_\infty(F)} G(F_v) \:=\: \prod_{v\in V_\infty(E)} G_0(E_v),
    \end{equation*}
     which we call the Lie group attached to $G$. Moreover, we fix a $\sigma$-stable maximal compact subgroup $K_\infty \subseteq G_\infty$. 
     Analogously, we define the Lie group $G_{0,\infty}$ attached to $G_0$.
     We obtain
       \begin{equation*}
        G_{0,\infty} \cong  \SL_2(\bbR)^s \times \SL_1(\bbH)^r,
      \end{equation*}
     where $s$ denotes the number of real places of $F$ where $D_0$ splits and $r$ denotes the number of real places ramified in $D_0$.
     The symbol $\bbH$ denotes Hamilton's division quaternion algebra over $\bbR$.
     
    Furthermore, we put $K_0 := \prod_{v \in V_f(F)} G_0(\cO_{F,v})$, which is an open compact subgroup of the locally compact group $G_0(\bbA_{F,f})$.
    Similarly, the group $K := \prod_{v \in V_f(F)} G(\cO_{F,v})$ is open and compact in $G(\bbA_{F,f})$.
     
   \subsection{Congruence subgroups}
    Let $\Ia_0 \subseteq \cO_F$ be a non-zero ideal.
    Let $v \in V_f(F)$ be a finite place. We obtain an open compact subgroup $K_{0,v}(\Ia_0)$ in $G_0(F_v)$
    defined by 
    \begin{equation*}
         K_{0,v}(\Ia_0) = \ker\bigl(G_0(\cO_{F,v}) \to G_0(\cO_{F,v}/\Ia_0\cO_{F,v})\bigr).
    \end{equation*}
    We also define
     \begin{equation*}
         K_{v}(\Ia_0) = \ker\bigl(G(\cO_{F,v}) \to G(\cO_{F,v}/\Ia_0\cO_{F,v})\bigr),
    \end{equation*}
    which is an open compact subgroup of $G(F_v)$. Putting this together we obtain the groups
    $K_0(\Ia_0) = \prod_{v \in V_f(F)} K_{0,v}(\Ia_0)$ and $K(\Ia_0) = \prod_{v \in V_f(F)} K_{v}(\Ia_0)$ which are open compact in $G_0(\bbA_{F,f})$ and
   $G(\bbA_{F,f})$ respectively.

\section{Lefschetz number of the Galois automorphism}\label{sectLefschetz}

\subsection{} In this section we will assume that the group scheme $G$ has strong approximation (cf. \cite[Thm. 4.3]{Vigneras1980}). This is the case precisely when there is at least one
  archimedean place $v \in V_\infty(E)$ of $E$ which splits the quaternion algebra $D$. Clearly, this always holds if $E$ has a complex place.

\subsection{}
  We fix a $\sigma$-stable maximal compact subgroup $K_\infty \subseteq G_\infty$. 
  The associated symmetric space 
 \begin{equation*}
            X = \bsl{K_\infty}{G_\infty}
  \end{equation*}
 is equipped naturally with an automorphism induced by $\sigma$.
  Let $\Gamma \subseteq G(F)$ be a torsion-free arithmetic subgroup. Such a group $\Gamma$ acts properly and freely on $X$ from 
  the right. The group cohomology
  $H^*(\Gamma, \bbC)$ is isomorphic to the cohomology $H^*(X/\Gamma, \bbC)$ of the locally symmetric space $X/\Gamma$.

   Assume further that $\Gamma$ is $\sigma$-stable, then $\sigma$ also induces an automorphism, again denoted by $\sigma$, of order two 
   on the space $X/\Gamma$.
   This automorphism induces maps in the cohomology $\sigma_q: H^q(X/\Gamma, \bbC) \to H^q(X/\Gamma, \bbC)$ in every degree $q$.
   We define the Lefschetz number of $\sigma$ as
   \begin{equation*}
       \calL(\sigma, \Gamma) \: := \: \sum_{q = 0}^\infty (-1)^q \tr(\sigma_q).
   \end{equation*}
   Since torsion-free arithmetic groups are of type (FL), this is a finite sum (of integers).

   We will apply a method developed by J. Rohlfs to compute such Lefschetz numbers (cf. \cite{Rohlfs1978}, \cite{Rohlfs1981}).
   The key observation is that the Lefschetz number of $\sigma$ equals the Euler characteristic of the space $(X/\Gamma)^\sigma$
   of $\sigma$-fixed points. Further, Rohlfs gave a precise description of the set of fixed points in terms of non-abelian
   Galois cohomology. We describe this fixed point decomposition in the adelic setting (as introduced in \cite{Rohlfs1990}).
   
 \subsection{} Let $\Ia_0 \subset \cO_F$ be a non-trivial proper ideal. We define the (principal) congruence subgroup of level $\Ia_0$
    in $G$ as
   \begin{equation*}
        \Gamma(\Ia_0) := \ker\bigl( G(\cO_F) \:\to\: G(\cO_F/\Ia_0) \bigr). 
   \end{equation*}
   Similarly, we define $\Gamma_0(\Ia_0) := \ker\bigl( G_0(\cO_F) \:\to\: G_0(\cO_F/\Ia_0) \bigr)$.
   We shall always assume that $\Ia_0$ was chosen sufficiently small such that these groups are torsion-free.
   This is the case, for instance, if $\Ia_0 \cap \bbZ$ is not a prime ideal of $\bbZ$. One should also notice that
   $\Gamma(\Ia_0) = G(F) \cap K(\Ia_0)$ and $\Gamma_0(\Ia_0) = G_0(F) \cap K_0(\Ia_0)$.
   We define $S(\Ia_0)$ to be the double quotient space
    \begin{equation*}
          S(\Ia_0) \: := \: \bsl{K_\infty K(\Ia_0)}{G(\bbA_F)}/G(F).
    \end{equation*}
    Using strong approximation we obtain a canonical homeomorphism
    \begin{equation*}
         X/\Gamma(\Ia_0) \: \stackrel{\simeq}{\longrightarrow} \: S(\Ia_0).
    \end{equation*}
    Note that $G(F)$ acts freely on the quotient space
    $\bsl{K_\infty K(\Ia_0)}{G(\bbA_F)}$ precisely when $\Gamma(\Ia_0)$ is torsion-free.
  
\subsection{Decomposition of the fixed point space}\label{sectFPDecomposition}
  We study the set $S(\Ia_0)^\sigma$ of $\sigma$-fixed points in the locally symmetric space $S(\Ia_0)$
  with the method of Rohlfs (see \cite{Rohlfs1990}).
  Suppose we are given an element $a \in G(\bbA_F)$ representing a $\sigma$-fixed double coset in $S(\Ia_0)$. 
  This means there are $k \in K_\infty K(\Ia_0)$ and $\gamma \in G(F)$ such that 
  \begin{equation}\label{eqGaloisRelationAdelic}
      \up{\sigma}a = k^{-1} a \gamma.
  \end{equation}
  The elements $k$ and $\gamma$ are uniquely determined by $a$ since $G(F)$ acts freely on $\bsl{K_\infty K(\Ia_0)}{G(\bbA_F)}$.
  Moreover, from $\up{\sigma\sigma}a = a$ one deduces the identities $k \up{\sigma}k = 1$ and  
  $\gamma \up{\sigma}\gamma = 1$. In other words, $k$ (resp. $\gamma$) defines
  a 1-cocycle in $Z^1(\sigma,K_\infty K(\Ia_0))$ (resp. $Z^1(\sigma, G(F))$). If one replaces $a$ by another
  representative $a'$ it is easily seen that the resulting cocycles are equivalent. 
  Consequently, a $\sigma$-fixed point in $S(\Ia_0)$ determines uniquely two cohomology classes: one in $H^1(\sigma, K_\infty K(\Ia_0))$
  and one in $H^1(\sigma, G(F))$. Moreover, equation \eqref{eqGaloisRelationAdelic} implies that these classes coincide, when
  they are mapped to $H^1(\sigma,G(\bbA_F))$ via the canonical maps induced by the respective embeddings.
  We define
  \begin{equation*}
        \calH^1(\Ia_0) := \: H^1(\sigma, K_\infty K(\Ia_0)) \mathop{\times}_{H^1(\sigma, G(\bbA_F))} H^1(\sigma, G(F))
  \end{equation*}
  as the fibred product of these two cohomology sets. One can show that this is in general a finite set. To see this, one 
  defines a surjective map $\alpha: H^1(\sigma, \Gamma(\Ia_0)) \to \calH^1(\Ia_0)$ and uses that
  the first set is finite due to a result of Borel and Serre (cf. Prop. 3.8 in \cite{BorelSerre1964}). 
  However, we will determine the set $\calH^1(\Ia_0)$ explicitly in \ref{sectH1}, thus we will not need this kind of general result.
  Summing up, we found a surjective map
   \begin{equation*}
        \vartheta: S(\Ia_0)^\sigma  \to \calH^1(\Ia_0).
   \end{equation*}
   Moreover, if we give the discrete topology on the finite set $\calH^1(\Ia_0)$, then this map is continuous.
   This means its fibres are open and closed in $S(\Ia_0)^\sigma$. 
   Hence the result is a topologically disjoint decomposition of the fixed point set 
   \begin{equation}\label{eqFixedPointDecomposition}
           S(\Ia_0)^\sigma \: = \: \bigsqcup_{\eta \in \calH^1(\Ia_0)} \vartheta^{-1}(\eta).
   \end{equation}

\subsection{Structure of fixed point components}\label{sectStructureFixedPoints}
  Rohlfs also gave a description of the fibres occuring in \eqref{eqFixedPointDecomposition} (cf. \cite{Rohlfs1990}).
  To describe them we need some more notation. Let $\gamma \in Z^1(\sigma,G(F))$ be
  a cocycle. By twisting $\sigma$ with $\gamma$ we obtain another automorphism $\sigma|\gamma$ on $G\times_{\cO_F}F$
  defined by $\sigma|\gamma := \inn(\gamma) \circ \sigma$. Here $\inn(\gamma)$ denotes the inner automorphism defined
  by $\gamma$. The group of $\sigma|\gamma$ fixed points is algebraic over $F$ and we denote it by $G(\gamma)$.
  Clearly, if $\gamma \in Z^1(\sigma, G(\cO_F))$, the twisted automorphism is defined over $\cO_F$ and so is $G(\gamma)$.
  Note that $G(1) = G^\sigma$.

  Moreover, if $k \in Z^1(\sigma, K_\infty K(\Ia_0))$ is a cocycle we can again twist the action of $\sigma$ on
   $K_\infty K(\Ia_0)$. The twisted action will be denoted $\sigma|k$ and its group of
  fixed points is written $(K_\infty K(\Ia_0))(k)$.

  Finally, we are able to describe the fibres of $\vartheta$. Let $\eta \in \calH^1(\Ia_0)$ be a class and choose
  representing cocycles $k \in Z^1(\sigma,K_\infty K(\Ia_0))$, $\gamma \in Z^1(\sigma, G(F))$ and 
  some $a \in G(\bbA_F)$ such that \eqref{eqGaloisRelationAdelic} holds.
  In this case there is a homeomorphism 
  \begin{equation*}
      \vartheta^{-1}(\eta) \stackrel{\simeq}{\longrightarrow} \bsl{a^{-1}(K_\infty K(\Ia_0))(k)a}{G(\gamma)(\bbA_F)}/G(\gamma)(F)
  \end{equation*}
  (cf. \cite[3.5]{Rohlfs1990}).   

\subsection{Determining $\calH^1$}\label{sectH1}
The description of the set of $\sigma$-fixed points followed a general pattern. In this subsection
 we start using specific properties of the involved groups.
Our first goal is to determine the set $\calH^1(\Ia_0)$ for a given ideal $\Ia_0 \subseteq \cO_F$.
We moved some of the purely local results we need to the appended Section \ref{secAppendixLocal}, 
since these results have a more technical flavour.

Let $R$ be any commutative $\cO_F$-algebra.
Whenever we write $H^1(\sigma, G(R)) = \{1\}$ we mean that $H^1$ consists of the trivial class only.
Moreover, the element $-1\in G(R)$ is always a cocycle for $\sigma$. We write $H^1(\sigma, G(R)) = \{\pm 1\}$
to express that $H^1$ consists of precisely two classes: the trivial class and a class represented by the cocycle $-1$.   

\begin{lemma}\label{lemmaH1LocalField}
    Let $v \in V_f(F)$ be a finite place, then $H^1(\sigma, G(F_v)) = \{1\}$.
\end{lemma}
\begin{proof}
   Note that we have $G(F_v) = G_0(F_v \otimes_{\cO_F} E )$. 
  We distinguish two cases with respect to the splitting behaviour of $v$ in $E$.

   First case: If $v$ splits in $E$, then $G(F_v) \cong G_0(F_v) \times G_0(F_v)$, and $\sigma$ acts by swapping
  the two components. Recall the following Lemma: Let $H$ be any group and denote the automorphism swapping 
  the two components in $H \times H$ by $\sigma$, then $H^1(\sigma, H\times H) =\{1\}$. To see this, one realizes that
   a cocycle in $H\times H$ is a pair $(x, x^{-1})$ with $x \in H$ arbitrary. However, $(x,x^{-1}) = (1, x)^{-1} (x,1)$ 
   is a trivial cocycle.

 Second case: If $v$ is not split, there is precisely one place $w \in V_f(E)$ lying over $v$ and
 \begin{equation*}
   G(F_v) \cong G_0(E_w) = \SL_1(D\otimes_E E_w).
 \end{equation*}
   In this case $\sigma$ acts by the nontrivial Galois automorphism of $E_w / F_v$ and the claim
   follows from Hilbert's Theorem 90 (cf. Corollary (29.4) in \cite[p.393]{BOInv}). 
\end{proof}

\begin{lemma}\label{lemmaH1LocalField2}
   Let $v \in V_\infty(F)$ be an infinite place of $F$. 
   If $v \in \Ram_\infty(D_0)$ and there is a complex $w \in V_\infty(E)$ of $E$ over $v$, then
   $H^1(\sigma, G(F_v)) = \{\pm 1\}$.
   In all other cases $H^1(\sigma, G(F_v)) = \{1\}$.
\end{lemma}
\begin{proof}
   Suppose there are two real places of $E$ over $v$. Then, as in \ref{lemmaH1LocalField}, we have an isomorphism
   $G(F_v) \cong G_0(F_v) \times G_0(F_v)$ where $\sigma$ acts by swapping the two components and
   the claim follows directly.

   Suppose now that there is a complex place $w \in V_\infty(E)$ lying over $v$. 
   By (29.2) in \cite[p.392]{BOInv} we have $H^1(\sigma, \GL_1(D_0\otimes_F E_w)) = \{1\}$ and 
   we get a short exact sequence
   \begin{equation*}
       1 \longrightarrow G(F_v) \longrightarrow \GL_1(D_0\otimes_F E_w) \stackrel{\nrd}{\longrightarrow} \bbC^\times \longrightarrow 1.
   \end{equation*}
   Consider the induced long exact sequence of pointed sets (cf. (28.3) in \cite{BOInv})
   \begin{equation*}
         1 \longrightarrow G_0(F_v) \longrightarrow \GL_1(D_0\otimes_F F_v) \stackrel{\nrd}{\longrightarrow} \bbR^\times \longrightarrow 
          H^1(\sigma, G(F_v)) \longrightarrow \{1\}.
   \end{equation*}
    If $D_0\otimes_F F_v$ is split, then the reduced norm is surjective and the claim follows. 
    Otherwise, suppose $v \in \Ram_\infty(D_0)$ then the image of the reduced norm only consists of the positive 
    real numbers and consequently $ H^1(\sigma, G(F_v))$ consists of two elements. It is easy to check that $1$
    and $-1$ are not equivalent.
\end{proof}

For an infinite place $v \in V_\infty(F)$ we denote the embedding $F \to F_v$ by $\iota_v$.
We say that an element of $F^\times$ is \emph{$D_0$-positive},
if for all $v \in \Ram_\infty(D_0)$ we have $\iota_v(x) > 0$  in $F_v \cong \bbR$. 
The multiplicative subgroup of $F^\times$ consisting of $D_0$-positive elements is denoted $F^\times_{D_0}$.
Similarly for $E$: an element $x \in E^\times$ is called $D$-positive, if $\iota_w(x) > 0$
for all $w \in \Ram_\infty(D)$. We write $E^\times_D$ for the group of $D$-positive elements.

Let $c$ denote the number of places $v \in \Ram_\infty(D_0)$ which are divided by a complex place of $E$.
There is an isomorphism $(E^\times_D\cap F)/F^\times_{D_0} \cong (\bbZ/2\bbZ)^c$. 

\begin{lemma}
There is a bijection between $H^1(\sigma, G(F))$ and $(E^\times_D\cap F)/F^\times_{D_0}$.
\end{lemma}
\begin{proof}
   As before, we have $H^1(\sigma, \GL_1(D)) = \{1\}$ (cf. (29.2) in \cite{BOInv}).
   By the theorem of Hasse-Schilling-Maass on norms (see Thm. 4.1 p.80 in \cite{Vigneras1980} for quaternion algebras 
or (33.15) in \cite{Reiner2003} for central simple algebras) the image of the reduced norm map
   $\nrd: \GL_1(D) \to E^\times$ is exactly $E^\times_D$. Similarly, we have $\nrd(\GL_1(D_0)) = F^\times_{D_0}$.
   Now, consider the exact sequence
   \begin{equation*}
       1 \longrightarrow G(F) \longrightarrow \GL_1(D) \stackrel{\nrd}{\longrightarrow} E^\times_D \longrightarrow 1.
   \end{equation*}
   As in the proof of Lemma \ref{lemmaH1LocalField2} there is a long exact sequence
   \begin{equation*}
         1 \longrightarrow G_0(F) \longrightarrow \GL_1(D_0) \stackrel{\nrd}{\longrightarrow} E^\times_D \cap F \longrightarrow 
          H^1(\sigma, G(F)) \longrightarrow \{1\}.
   \end{equation*}
\end{proof}
\begin{corollary}\label{corBijGlobalInfinity}
 The canonical map $H^1(\sigma, G(F)) \to H^1(\sigma, G_\infty)$ is bijective.
\end{corollary}
\begin{remark}\label{remBijInfinity}
   The canonical map $H^1(\sigma, K_\infty) \to H^1(\sigma, G_\infty)$ is a bijection.
   This follows in general for connected semi-simple groups by an argument of Rohlfs using the Cartan decomposition.
   The reader may consult, for example, Lemma 1.4 in \cite{Rohlfs1981}.
\end{remark}

\begin{lemma}\label{lemLocalH1Summary}
   Let $v \in V_f(F)$ be a finite place. If $v$ is unramified in $E$, then $H^1(\sigma, G(\cO_{F,v})) = \{1\}$.
   If $v$ ramifies in $E$ and lies over an odd prime number, then $H^1(\sigma, G(\cO_{F,v})) = \{\pm 1\}$.
\end{lemma}
\begin{proof}
   Let $\Ip_0$ be the prime ideal corresponding to $v$.
   In the case where $\Ip_0$ is split in $E$, the claim follows as in Lemma \ref{lemmaH1LocalField} 
   since $G(\cO_{F,v}) \cong G_0(\cO_{F,v}) \times G_0(\cO_{F,v})$.
   
   The other cases are treated in Corollary \ref{corLocalUnramified1} and Lemma \ref{lemLocalRamified1} in the appendix.
\end{proof}

\begin{corollary}
  The canonical map $H^1(\sigma, G(F)) \to H^1(\sigma, G(\bbA_F))$ is bijective.
  In particular, the projection $\calH^1(\Ia_0) \to H^1(\sigma, K_\infty K(\Ia_0))$ is a bijection for every non-trivial proper
 ideal $\Ia_0 \subseteq \cO_F$.
\end{corollary}
\begin{proof}
  Notice that $H^1(\sigma,G(\bbA_F)) = H^1(\sigma, G_\infty) \times H^1(\sigma, G(\bbA_{F,f}))$.
  It follows from Lemma \ref{lemmaH1LocalField} and Lemma \ref{lemLocalH1Summary} that $H^1(\sigma, G(\bbA_{F,f})) = \{1\}$.
  Thus the result follows from Corollary \ref{corBijGlobalInfinity}. 
\end{proof}

Let $S$ be the set of finite places $v \in V_f(F)$ which divide $2$ and which are ramified in $E$.
This is the set of places where determining the local $H^1$ is difficult (see also Remark \ref{remRohlfsGeneral}).
We define $K(\Ia_0,2) := \prod_{v \in S} K_v(\Ia_0)$.
Moreover, let $R$ be the set of finite places $v \in V_f(F)$ which are ramified in $E$ but which do not divide $2$.
Given an ideal $\Ia_0 \subseteq \cO_F$, we define 
$\rho(\Ia_0) := |\{\:v \in R\:|\:v \text{ does not divide } \Ia_0\:\}|$.
As above, let $c$ be the number of places $v \in \Ram_\infty(D_0)$ which are divided by a complex place of $E$.
We get the following corollary.

\begin{corollary}\label{corSizeOfH1}
 For every non-trivial proper ideal $\Ia_0 \subseteq \cO_F$ the set $\calH^1(\Ia_0)$ consists of
  \begin{equation*}
      2^{c+\rho(\Ia_0)} |H^1(\sigma,K(\Ia_0, 2))|
  \end{equation*}
 elements.
\end{corollary}
\begin{proof}
  The assertion follows from the bijection $\calH^1(\Ia_0) \to H^1(\sigma, K_\infty K(\Ia_0))$ together with Remark \ref{remBijInfinity} and
  the local results Lemma \ref{lemLocalUnramified2} and Lemma \ref{lemLocalRamified2} which can be found in the appendix.
\end{proof}

\subsection{Euler Characteristic of Fixed Point Components}
In this section we compute the Euler characteristic of the fixed point components $\vartheta^{-1}(\eta)$ defined in \ref{sectFPDecomposition}.
Let $\Ia_0 \subseteq \cO_F$ be a non-trivial ideal.  We choose a class 
$\eta \in \calH^1(\Ia_0)$ and a
representative $(k, \gamma) \in Z^1(\sigma, K_\infty K(\Ia_0))\times Z^1(\sigma, G(F))$ together with $a \in G(\bbA_F)$
which satisfies $\up{\sigma}a = k^{-1}a\gamma$. Since we still assume $G$ to have strong approximation, we can achieve that 
$a \in G_\infty$ and $\gamma \in \Gamma(\Ia_0)$ (changing the chosen representative).
Then the group $G(\gamma)$ of fixed points of the $\gamma$-twisted action is a group scheme defined over $\cO_F$.

\begin{remark}\label{remFixedPointCompOfSameType}
  For any $\gamma \in Z^1(\sigma, G(F))$ 
  the fixed point group $G(\gamma)$ and $G_0$ are isomorphic over $F$.

 This can be seen as follows: By Hilbert's Theorem 90 we have $H^1(\sigma, \GL_1(D)) = \{1\}$. Moreover,
 the canonical map $\inn_*: H^1(\sigma, G(F)) \to H^1(\sigma, \Aut_F(G))$ factors through $H^1(\sigma, \GL_1(D))$ and thus is trivial.
  We deduce the existence of an automorphism $\psi: G\times_{\cO_F} F \to G\times_{\cO_F} F$ such that
 \begin{equation*}
   \inn(\gamma) = \psi^{-1} \circ \sigma \circ \psi \circ \sigma^{-1}.
 \end{equation*}
 In other words, $\psi$ is an isomorphism of $G$ over $F$ such that $\psi \circ \sigma|\gamma = \sigma \circ \psi$.
Recall that $G_0 \cong G^\sigma$ over $F$.

 We deduce that the fixed point components $\vartheta^{-1}(\eta)$ are all associated to the same group over $F$ (cf. \ref{sectStructureFixedPoints}). 
 An important consequence is that the sign of the Euler characteristic $\chi(\vartheta^{-1}(\eta))$
 is the same for all the components. This can be seen as follows. First note that
 Harder's Gauß-Bonnet theorem (see \cite{Harder1971}) implies that we may use the Euler-Poincar\'e measure (in the sense of Serre)
 to compute the Euler characteristic.
 Further, the sign of the Euler-Poincar\'e measure
  only depends on the structure of the associated real Lie group (see Prop. 23 in \cite{Serre1971}). 
\end{remark}

\begin{theorem}\label{ThmEulerCharacteristic}
  Let $\Ia_0 \subseteq \cO_F$ be a proper ideal (such that $\Gamma(\Ia_0)$ is torsion-free) and
  let $K_{0,\infty}$ be any maximal compact subgroup of $G_{0,\infty}$.
  Then the Euler characteristic of the double coset space $\bsl{K_{0,\infty} K_0(\Ia_0)}{G_0(\bbA_F)}/G_0(F)$ 
   can be computed using the following formulas
  \begin{align*}
       \chi\bigl(&\bsl{K_{0,\infty} K_0(\Ia_0)}{G_0(\bbA_F)}/G_0(F)\bigr) \\
              &= 
                     (-1/2)^r \zeta_F(-1) [K_0: K_0(\Ia_0)]\prod_{\Ip_0 \in \Ram_f(D_0)} (\Nr_{F/\bbQ}(\Ip_0)-1) \\
                &= (-2)^s (4\pi^2)^{-[F:\bbQ]}\zeta_F(2) |\disc_F|^{3/2} [K_0: K_0(\Ia_0)]\prod_{\Ip_0 \in \Ram_f(D_0)} (\Nr_{F/\bbQ}(\Ip_0)-1).\\
   \end{align*}
  Here $r$ denotes the number of real places of $F$ ramified in $D_0$ and $s$ denotes the number real places where $D_0$ splits.
  Moreover, $\zeta_F$ denotes the zeta function of the number field $F$, $\disc_F$ denotes the discriminant of $F$ and
  $\N_{F/\bbQ}(\Ip_0):= |\cO_F/\Ip_0|$ denotes the ideal norm.
\end{theorem}
\begin{proof}
  Since $F$ is totally real, the functional equation of the zeta function implies
\begin{equation*}
   \zeta_F(2) |\disc_F|^{3/2} (2\pi^2)^{-[F:\bbQ]} = (-1)^{[F:\bbQ]} \zeta_F(-1).
\end{equation*}
So, the first equality is an immediate consequence of the second.

For simplicity we write 
\begin{equation*}
  S_0(\Ia_0) := \bsl{K_{0,\infty} K_0(\Ia_0)}{G_0(\bbA_F)}/G_0(F).
\end{equation*}
We will distinguish whether $G_0$ has strong approximation or not. This is not absolutely necessary, but it stresses the difference of these two cases.

If $G_0$ has strong approximation, then $G_{0,\infty}$ is not compact and $S_0(\Ia_0)$ is
homeomorphic to the locally symmetric space $X_0/\Gamma_0(\Ia_0)$, where $X_0 := \bsl{K_{0,\infty}}{G_{0,\infty}}$.
The Euler-Poincar\'e measure $\vol_\chi$ (in the sense of Serre cf. \cite{Serre1971}) on $G_{0,\infty}$ is given by 
\begin{equation*}
\vol_\chi =  (-2)^s (4\pi^2)^{-[F:\bbQ]} \vol_T,
\end{equation*} 
 where $\vol_T$ denotes the Tamagawa measure
on $G_{0,\infty}$ as defined in \cite[p.54]{Vigneras1980} or \cite[p.242]{MaclachlanReid2003}.
Using strong approximation and the assumption that $\Gamma(\Ia_0)$ (and hence $\Gamma_0(\Ia_0)$) is torsion-free, we find
\begin{align*}
 \chi(S_0(\Ia_0)) &= \chi(\Gamma_0(\Ia_0)) = \vol_\chi(G_{0,\infty}/\Gamma_0(\Ia_0))\\
                 &= (-2)^s(4\pi^2)^{-[F:\bbQ]} \vol_T(\bsl{K_0(\Ia_0)}{G_0(\bbA_F)}/G_0(F))\\
                 &= (-2)^s(4\pi^2)^{-[F:\bbQ]} \vol_T(K_0(\Ia_0))^{-1}\\
                 &= (-2)^s(4\pi^2)^{-[F:\bbQ]} [K_0:K_0(\Ia_0)] \vol_T(K_0)^{-1}.
\end{align*}
Here we used that the Tamagawa number $\vol_T(G_0(\bbA_F)/G_0(F))$ is one (cf. \cite[2.3, p.71]{Vigneras1980} or \cite[Thm. 7.6.3]{MaclachlanReid2003}).
It is known that
 \begin{equation*}
  \vol_T(K_0)^{-1} = \zeta_F(2) |\disc_F|^{3/2} \prod_{\Ip_0 \in \Ram_f(D_0)} (\Nr_{F/\bbQ}(\Ip_0)-1),
 \end{equation*} 
  the reader may consult Vign\'eras' book \cite[p.55]{Vigneras1980}.

Assume now, that $G_{0,\infty}$ is compact, i.e. $r= [F:\bbQ]$.
In this case $G_0(\bbA_F)$ is a finite union
\begin{equation*}
    G_0(\bbA_F) = \bigsqcup_{i=1}^m G_{0,\infty}K_0(\Ia_0)x_i G_0(F)
\end{equation*}
for some $x_1, \dots, x_m \in G_0(\bbA_F)$. Note, that the assumption that $\Gamma(\Ia_0)$ is torsion-free implies that 
$G_{0,\infty} K_0(\Ia_0)$ acts freely on $G_0(\bbA_F)/G_0(F)$.
Further, 
$S_0(\Ia_0)$ consists precisely of $m$ points, so $\chi(S_0(\Ia_0)) = m$ and we only have to compute 
this number. As before,
\begin{align*}
   m = \vol_T(S_0(\Ia_0)) &= \vol_T(G_{0,\infty})^{-1}\vol_T(K_0(\Ia_0))^{-1} \\
   &=(4\pi^2)^{-r} [K_0:K_0(\Ia_0)] \vol_T(K_0)^{-1},
\end{align*}
hence the claim follows.
\end{proof}

\begin{corollary}\label{corEulerCharacteristic}
  Let $K_f \subseteq G_0(\bbA_{F,f})$ be an open compact subgroup, which has the same invariant volume as $K_0(\Ia_0)$, this means 
  $\vol_T(K_f) = \vol_T(K_0(\Ia_0))$.
  If, moreover, $K_{0,\infty} K_f$ acts freely on $G_0(\bbA_F)/G_0(F)$, then
  the formulas of Theorem \ref{ThmEulerCharacteristic} also hold for the Euler characteristic
   \begin{equation*}
       \chi\bigl(\bsl{K_{0,\infty} K_f}{G_0(\bbA_F)}/G_0(F)\bigr). 
   \end{equation*}
\end{corollary}
\begin{proof}
   The only two important assumptions on $K_0(\Ia_0)$ that we used in the proof of Theorem \ref{ThmEulerCharacteristic} is that
   $K_{0,\infty} K_0(\Ia_0)$ acts freely on $G_0(\bbA_F)/G_0(F)$ and the formula for the volume of $K_0(\Ia_0)$ with 
   respect to the Tamagawa measure.
\end{proof}

\subsection{The Lefschetz number}

In this section we finally compute the Lefschetz number $\calL(\sigma,\Gamma(\Ia_0))$ of $\sigma$ on the locally symmetric space $X/\Gamma(\Ia_0) \cong S(\Ia_0)$.
Recall the following theorem
\begin{theorem}
  If $\Gamma(\Ia_0)$ is torsion-free, then
   \begin{equation*}
       \calL(\sigma,\Gamma(\Ia_0)) = \chi(S(\Ia_0)^\sigma).
   \end{equation*}
\end{theorem}
This kind of Lefschetz fixed point principle has been observed by many people. In the context of arithmetic groups this theorem
is due to Rohlfs (see, for instance, \cite[Prop. 1.9]{Rohlfs1981}).
The theorem can be proven either by adapting the proof of 1.9 in \cite{Rohlfs1981} or
by an application of Cor. 7.15 in \cite{Kionke2012}. 

\begin{definition}
   We say that the extension $E/F$ of number fields is \emph{unramified over $2$},
   if for every pair of finite places $v \in V_f(F)$, $w \in V_f(E)$ with $w | v$ and $v| 2$
   the extension $E_w / F_v$ is unramified.
\end{definition}

To shorten the notation we define $\Delta(D_0) := \prod_{\Ip_0 \in \Ram_f(D_0)} (\Nr_{F/\bbQ}(\Ip_0)-1)$ and we write
 $d = [F:\bbQ]$.

\begin{theorem}\label{thmLefschetzNumber}
Suppose that $G$ has strong approximation.
Let $\Ia_0 \subseteq \cO_F$ be a non-trivial ideal such that $\Gamma(\Ia_0)$ is torsion-free.
The sign of the Lefschetz number $\calL(\sigma,\Gamma(\Ia_0))$ is $(-1)^s$ where $s$ is the number of real places of $F$ which split $D_0$.
Moreover, the Lefschetz number can be bounded from below by
\begin{equation*}
      |\calL(\sigma,\Gamma(\Ia_0))| \geq 2^{c+\rho(\Ia_0)-r -d} \pi^{-2d}\zeta_F(2)|\disc_F|^{3/2} \Delta(D_0) [K_0:K_0(\Ia_0)].
\end{equation*}
If $E/F$ is unramified over $2$, there is the exact formula
\begin{equation*}
 \calL(\sigma,\Gamma(\Ia_0)) = (-1)^s 2^{c+\rho(\Ia_0)-r -d} \pi^{-2d}\zeta_F(2)|\disc_F|^{3/2} \Delta(D_0) [K_0:K_0(\Ia_0)].
\end{equation*}
The numbers $c$ and $\rho(\Ia_0)$ are defined as in Corollary \ref{corSizeOfH1}.
\end{theorem}
\begin{proof}
  The Euler characteristic is additive for topologically disjoint unions, so 
 \begin{equation*}
   \calL(\sigma,\Gamma(\Ia_0)) = \sum_{\eta \in \calH^1(\Ia_0)} \chi(\vartheta^{-1}(\eta)).
 \end{equation*}
 As pointed out in Remark \ref{remFixedPointCompOfSameType} the sign of the Euler characteristic $\chi(\vartheta^{-1}(\eta))$ is the same
  for all the components $\vartheta^{-1}(\eta)$ . Thus, to obtain an estimate for the Lefschetz number, it suffices
  to calculate $\chi(\vartheta^{-1}(\eta))$ for
  all $\eta$ in some chosen subset $T \subseteq \calH^1(\Ia_0)$. Let $q: \calH^1(\Ia_0) \to H^1(\sigma,K(\Ia_0,2))$ denote the canonical map 
  (the definition of $K(\Ia_0,2)$ can be found in the paragraph preceding Corollary \ref{corSizeOfH1}).
 Define 
\begin{equation*}
T:= \{\: \eta \in \calH^1(\Ia_0) \:|\: q(\eta) = 1 \:\}.
\end{equation*}
 From the definition of $K(\Ia_0,2)$ it is clear that
 $T = \calH^1(\Ia_0)$ if $E/F$ is unramified over $2$.
 Let $\eta \in T$ and choose a representative $(k_\infty k , \gamma) \in K_\infty K(\Ia_0) \times G(F)$ with $a \in G(\bbA_F)$ satisfying
   $\up{\sigma}a = k_\infty^{-1}k^{-1}a \gamma$. Using strong approximation, we can choose $\gamma$, $k$ and $a$ such that
  $\gamma \in G(F) \cap  K(\Ia_0) = \Gamma(\Ia_0)$, $k \in K(\Ia_0)$ and
  $a \in G_\infty$. Then $G(\gamma)$ is defined over $\cO_F$. 
  Again Hilbert 90 yields an element $b \in \GL_1(D)$ such that $\gamma = b^{-1}\up{\sigma}b$. The conjugation $\inn(b)$ with $b$ defines an
  $F$-isomorphism $G(\gamma)\times_{\cO_F} F \to G_0 \times_{\cO_F} F$. 
  Define $K_f := \inn(b)(K(\Ia_0)(k))$ which is open compact in $G_0(\bbA_{F,f})$ and define $K'_{0,\infty} := \inn(b)(a^{-1}K_\infty(k_\infty) a)$ 
  which is maximal compact in $G_{0,\infty}$.
  Furthermore, $\inn(b)$ induces a homeomorphism
  \begin{equation*}
      \vartheta^{-1}(\eta) \stackrel{\simeq}{\longrightarrow}
                    \bsl{K'_{0,\infty} K_f}{G_0(\bbA_F)}/G_0(F).
  \end{equation*}
  Note, that $K'_{0,\infty} K_f$ acts freely on $G_0(\bbA_F)/G_0(F)$ due to the assumption that $\Gamma(\Ia_0)$ is torsion-free.
  Eventually, we have to check that $K_f$ has the same invariant volume as $K_0(\Ia_0)$ to use Corollary \ref{corEulerCharacteristic}.

  How to compare these two volumes?  Let $v$ be any finite place of $F$. By the choice of $T$ and the local
  results \ref{corLocalUnramified1} and \ref{lemLocalRamified2}, we find
  $z_v \in G(\cO_{F,v})$ such that $\gamma = \pm z_v^{-1} \up{\sigma}z_v$. Therefore, conjugation
  with $z_v$ yields an isomorphism of topological groups $\inn(z_v): G(\gamma)(F_v) \to G_0(F_v)$ 
  mapping $K_v(\Ia_0)(\gamma)$ to $K_{0,v}(\Ia_0)$. We compose this isomorphism with $\inn(b^{-1})$ obtained before and get
   \begin{equation*}
       \inn(z_vb^{-1}) : G_0(F_v) \to G_0(F_v).
   \end{equation*}
   One can verify that this automorphism is unimodular, using that it is the conjugation by some
   element in the larger group $\GL_1(D_0 \otimes F_v)$.
\end{proof}

\section{Estimates}\label{sectEstimates}

\subsection{} Let $\Ia_0 \subseteq \cO_F$ be a non-trivial ideal. The purpose of this section is to provide simple estimates for
the ratio $[K_0:K_0(\Ia_0)]/\sqrt{[K:K(\Ia_0)]}$.

 Using the smoothness of the scheme $G_0$ we see that
\begin{equation*}
    [K_0 : K_0(\Ia_0)] = \prod_{ v | \Ia_0} |G_0(\cO_{F,v}/\Ia_0\cO_{F,v})|,
\end{equation*}
and similarly smoothness of $G$ yields
 \begin{equation*}
    [K : K(\Ia_0)] = \prod_{ v | \Ia_0} |G(\cO_{F,v}/\Ia_0\cO_{F,v})|.
\end{equation*}
We compare the terms $|G_0(\cO_{F,v}/\Ia_0\cO_{F,v})|$ and $|G(\cO_{F,v}/\Ia_0\cO_{F,v})|$, but we will have
to consider different cases according to the splitting behaviour. 
We choose some prime ideal $\Ip_0$ dividing $\Ia_0$ and take $e$ to be maximal with the property $\Ip_0^e | \Ia_0$.
The finite place of $F$ corresponding to $\Ip_0$ will be denoted $v$.
Define 
\begin{equation*}
 Q(v, \Ia_0) =\frac{|G_0(\cO_{F,v}/\Ia_0\cO_{F,v})|}{\sqrt{|G(\cO_{F,v}/\Ia_0\cO_{F,v})|}}.
\end{equation*}
Moreover, we write $\Nr(\Ip_0) = |\cO_F/\Ip_0|$ for the norm of the prime ideal.

\subsection{Case: $\Ip_0$ splits in $E$} 
Suppose that $\Ip_0$ splits in $E$, then $\Ip_0\cO_E = \IP \IQ$ where $\IP$ and $\IQ$ are distinct prime ideals in $\cO_E$.
In this case 
\begin{equation*}
G(\cO_{F,v}/\Ia_0\cO_{F,v}) = G(\cO_F/\Ip_0^e) = G_0(\cO_E/\IP^e\IQ^e) \cong  G_0(\cO_F/\Ip_0^e) \times G_0(\cO_F/\Ip_0^e).
\end{equation*}
Consequently, $|G(\cO_{F,v}/\Ia_0)| = |G_0(\cO_{F,v}/\Ia_0\cO_{F,v})|^2$ and hence $Q(v, \Ia_0) =1$.

\subsection{Case: $\Ip_0$ is inert in $E$}
Suppose that $\Ip_0$ is inert in $E$, this means $\Ip_0\cO_E = \IP$ is a prime ideal in $\cO_E$. In this case 
the local extension is unramified. According to Lemma \ref{lemLocalCalculationIndex}
we get 
\begin{equation*}
    Q(v, \Ia_0)^2 =  (1-\Nr(\Ip_0)^{-2})(1+\Nr(\Ip_0)^{-2})^{-1}
\end{equation*}
if $D_0$ splits at $\Ip_0$. Whereas, 
\begin{equation*}
   Q(v, \Ia_0)^2 =  (1+\Nr(\Ip_0)^{-1})(1-\Nr(\Ip_0)^{-1})^{-1}
\end{equation*}
if $D_0$ ramifies in $v$. Notice that in the latter case $Q(v,\Ia_0) > 1$.

\subsection{Case: $\Ip_0$ is ramified in $E$}
Assume that $\Ip_0$ is ramified in $E$. In this case $\Ip_0\cO_E = \IP^2$ for some prime ideal $\IP \subset \cO_E$.
The local extension is ramified and we obtain
\begin{equation*}
    Q(v, \Ia_0)^2 =  \begin{cases}
                        1-\Nr(\Ip_0)^{-2}              \:&\: \text{ if $D_0$ splits at $v$} \\
                        1+\Nr(\Ip_0)^{-1}              \:&\: \text{ if $D_0$ ramified at $v$}.
                     \end{cases}
\end{equation*}
Notice that $v \in \Ram_f(D_0)$ implies $Q(v,\Ia_0) > 1$.

\subsection{Results} One can use these three cases to derive a formula 
for the quotient $[K_0:K_0(\Ia_0)]/\sqrt{[K:K(\Ia_0)]}$. However, this will not be important for our purposes.
We content ourselves with the following corollary.
\begin{corollary}\label{corEstimate}
 Let $\Ia_0 \subset \cO_F$ be a non-trivial ideal, then
   \begin{equation*}
    \frac{[K_0:K_0(\Ia_0)]}{\sqrt{[K:K(\Ia_0)]}} \geq \zeta_F(2)^{-1}.
 \end{equation*}
 Suppose that all prime ideals dividing $\Ia_0$ are either split in $E$ or are ramified in $D_0$,
 then
 \begin{equation*}
    \frac{[K_0:K_0(\Ia_0)]}{\sqrt{[K:K(\Ia_0)]}} \geq 1.
 \end{equation*}
\end{corollary}
\begin{proof}
The second assertion follows directly from what we have seen before.
To prove the first, we start with an estimation replacing all terms that are at least one by terms which are smaller than one.
One obtains
\begin{align*}
  \frac{[K_0:K_0(\Ia_0)]^2}{[K:K(\Ia_0)]} & = \prod_{ v | \Ia_0} Q(v,\Ia_0)^2 \\
   &\geq \prod_{\substack{\Ip_0 | \Ia_0 \\ \Ip_0 \text{ inert}}} (1-\Nr(\Ip_0)^{-2})(1+\Nr(\Ip_0)^{-2})^{-1}
   \prod_{\substack{\Ip_0 | \Ia_0 \\ \Ip_0 \text{ ramified}}}   (1-\Nr(\Ip_0)^{-2}) \\
   & \geq \zeta_F(2)^{-1} \prod_{\substack{\Ip_0 | \Ia_0 \\ \Ip_0 \text{ inert}}}(1+\Nr(\Ip_0)^{-2})^{-1} \\
   & \geq \zeta_F(2)^{-1} \prod_{\Ip_0}(1+\Nr(\Ip_0)^{-2}+\Nr(\Ip_0)^{-4} + \dots)^{-1}  \: \geq \zeta_F(2)^{-2}.
\qedhere
\end{align*}
\end{proof}

\section{Application to Hyperbolic 3-Manifolds}\label{sectHyperbolicApplications}

\subsection{Assumptions}\label{secSetting}
For this section we fix the following assumptions. As before $F$ is a totally real number field, and we define $d= [F:\bbQ]$.
Choose once and for all a real place $v_0$ of $F$. Let $E/F$ be a quadratic extension such that $E$ has precisely one complex place $w_0$, further
assume $w_0 | v_0$. Moreover, let $D_0$ be an $F$ quaternion division algebra such that $V_\infty(F) \setminus \{v_0\} \subseteq \Ram_\infty(D_0)$.
This means $D_0$ is ramified in every real place of $F$ except possibly $v_0$. Then $D := D_0 \otimes_F E$ 
satisfies $\Ram_\infty(D) = V_\infty(E)\setminus\{w_0\}$. We will assume that $D$ is a division algebra. This assumption is implied by
the previous assumptions if $d= [F:\bbQ]$ is at least two.

 The number $s$ of real places of $F$ that split $D_0$ is $s= 0$ if $v_0 \in \Ram_\infty(D_0)$ and otherwise $s=1$.
As always $r = d - s$ and therefore the number $c$ of places in $\Ram_\infty(D_0)$ that are divided by a complex place in $E$ is exactly $c=1-s$.

\subsection{}
The real Lie group $G_\infty$ is isomorphic to
\begin{equation*}
    G_\infty \cong \SL_2(\bbC) \times \SL_1(\bbH)^{2(d-1)}.
\end{equation*}
 The group scheme $G$ has strong approximation, since the group $G_\infty$ is not compact.
 Given a non-trivial ideal $\Ia_0 \subset \cO_F$, the group $\Gamma(\Ia_0)$
 embeds discretely into $G_\infty$. The assumption that $D$ is a division algebra implies that 
 $\Gamma(\Ia_0)$ is cocompact in $G_\infty$ (cf. Thm. 8.2.3 in \cite{MaclachlanReid2003} or more general \cite[Thm. 8.4]{Borel1969} ).
 Moreover, the projection $G_\infty \to \SL_2(\bbC)$ is proper and open, thus
 $\Gamma(\Ia_0)$ projects isomorphically to a discrete cocompact subgroup of $\SL_2(\bbC)$. Fix a
 maximal compact and $\sigma$-stable subgroup $K_\infty \subseteq G_\infty$.
 The symmetric space $X:= \bsl{K_\infty}{G_\infty}$ is isomorphic to hyperbolic three space $\Hyp$.
 Suppose that $\Gamma(\Ia_0)$ is torsion-free, then $\Gamma(\Ia_0)$ is a cocompact Kleinian group
 and $X/\Gamma(\Ia_0) \cong \Hyp/\Gamma(\Ia_0)$ is a compact orientable hyperbolic manifold.

\subsection{A general remark} Let $M$ be a closed connected smooth oriented manifold, say of odd dimension $\dim(M) = n = 2m+1$.
Let $\tau: M \to M$ be a smooth automorphism of $M$ of such that $\tau^2 = \id_M$. 
We consider the de Rham cohomology of $M$ with complex
coefficients and the non-degenerate Poincar\'e pairing
\begin{equation*}
  \langle\cdot,\cdot\rangle: H^j(M, \bbC) \times H^{n-j}(M, \bbC) \to \bbC
\end{equation*}
for $0\leq j \leq n$. Let $\tau_j: H^j(M, \bbC)\to H^j(M, \bbC)$ denote
the induced automorphism in the cohomology in degree $j$. 
For classes $\alpha \in H^j(M, \bbC)$ and $\beta \in H^{n-j}(M, \bbC)$
we have $\langle\tau_j(\alpha),\tau_{n-j}(\beta)\rangle = \epsilon \langle\alpha,\beta\rangle$,
with $\epsilon=1$ if $\tau$ is orientation preserving and $\epsilon = -1$ otherwise.
Let $H^j(M, \bbC) = H^j_{1} \oplus H^j_{-1}$ be the eigenspace decomposition with respect to $\tau_j$.

If $\tau$ preserves orientation, then $H^j_{1} \perp H^{n-j}_{-1}$ and $H^j_{-1} \perp H^{n-j}_{1}$ for all $0\leq j \leq n$.
Consequently, $\dim(H^j_{1}) = \dim(H^{n-j}_{1})$ and $\dim(H^j_{-1}) = \dim(H^{n-j}_{-1})$. Under the assumption that $\dim(M) = n$ is odd, this implies that the
 Lefschetz number $\calL(\tau,M)$ vanishes.
In particular, we deduce: $\calL(\tau,M) \neq 0$ implies that $\tau$ is not orientation preserving. 

Assume now that $\tau$ changes the orientation. In this case $H^j_{1} \perp H^{n-j}_{1}$ and $H^j_{-1} \perp H^{n-j}_{-1}$ and
 therefore $\dim(H^j_{1}) = \dim(H^{n-j}_{-1})$ and $\dim(H^j_{-1}) = \dim(H^{n-j}_{1})$.
Consequently, the following formula gives the Lefschetz number of $\tau$
\begin{equation*}
  \calL(\tau, M) = 2\sum_{j=0}^m (-1)^j (\dim(H^j_1)-\dim(H^j_{-1})),
\end{equation*}
where $m = (\dim(M)-1)/2$.
Specializing to the case $\dim(M) = 3$ we obtain
\begin{equation}\label{eqLefschetzOrient}
  \calL(\tau, M) = 2 - 2\dim(H^1_1) + 2\dim(H^1_{-1}).
\end{equation}

\subsection{A lower bound for the first Betti number}
 We go back to the setting introduced in \ref{secSetting}.
Let $\Ia_0 \subseteq \cO_F$ be a non-trivial ideal such that $\Gamma(\Ia_0)$ is torsion-free.
Recall that we defined 
\begin{equation*}
 \rho(\Ia_0) := |\{\:\Ip_0 \subset \cO_F \:|\:\Ip_0 \text{ prime ideal ramified in $E$ and } \: \Ip_0 \nmid 2\Ia_0 \:\}|.
\end{equation*}
Theorem \ref{thmLefschetzNumber} yields
\begin{equation} \label{eqLefschetzSpecial}
   |\calL(\sigma, \Hyp/\Gamma(\Ia_0))| \geq 2^{1+\rho(\Ia_0)} (2\pi)^{-2d} \zeta_F(2)|\disc_F|^{3/2} \Delta(D_0)[K_0:K_0(\Ia_0)].
\end{equation}
Moreover, the Lefschetz number $\calL(\sigma, \Hyp/\Gamma(\Ia_0))$ is negative if $s = 1$ and positive otherwise.
Clearly, the Lefschetz number is not zero and we deduce that $\sigma$ changes the orientation on $\Hyp/\Gamma(\Ia_0)$.
We use this to estimate the size of the first Betti number.
\begin{theorem}\label{thmH1Estimate}
   In the notation introduced above
   \begin{equation*}
      \dim(H^1(\Gamma(\Ia_0), \bbC)) \geq 2^{\rho(\Ia_0)} (2\pi)^{-2d} \zeta_F(2)|\disc_F|^{3/2} \Delta(D_0)[K_0:K_0(\Ia_0)] + (-1)^{s+1}.
   \end{equation*}
\end{theorem}
\begin{proof}
   It follows from equation \eqref{eqLefschetzOrient} that
  \begin{equation*}
      (-1)^s(1/2)|\calL(\sigma,\Gamma(\Ia_0))|-1 = \dim(H^1_{-1})-\dim(H^1_{1})
  \end{equation*}
    Multiply with $(-1)^s$ (the sign of the Lefschetz number), plug in the right hand side of \eqref{eqLefschetzSpecial} and the claim follows.
\end{proof}
\begin{remark}
  Theorem \ref{thmH1Estimate} implies directly that the first Betti number may become arbitrarily large as $\Ia_0$ varies, since the 
  term $[K_0:K_0(\Ia_0)]$ is unbounded. Moreover, only the term $[K_0:K_0(\Ia_0)]$ is responsible for the order of growth, since
  $2^{\rho(\Ia_0)}$ is bounded by some number depending on the extension $E/F$.
\end{remark}

Let $\Gamma(1) := G(\cO_F) = \SL_1(\Lambda)$ be the norm one group of the order $\Lambda$. For every non-trivial
ideal $\Ia_0 \subset \cO_F$, the
index $[\Gamma(1):\Gamma(\Ia_0)]$ satisfies
\begin{equation*}
    [\Gamma(1):\Gamma(\Ia_0)] = [K: K(\Ia_0)].
\end{equation*}
This can be checked exploiting strong approximation of the group $G$.

\begin{corollary}\label{corEstimateH1Index}
 For every non-trivial ideal $\Ia_0 \subseteq \cO_F$ such that $\Gamma(\Ia_0)$ is torsion-free
 the following holds
  \begin{equation*}
    \dim(H^1(\Gamma(\Ia_0), \bbC))+ (-1)^s \geq 2^{\rho(\Ia_0)} (2\pi)^{-2d}|\disc_F|^{3/2} \Delta(D_0)[\Gamma(1):\Gamma(\Ia_0)]^{1/2}.
  \end{equation*}
   In particular, there is a positive real number $\kappa(F,D_0)$ such that
    \begin{equation*}
        \dim(H^1(\Gamma(\Ia_0), \bbC)) \geq \kappa(F,D_0)[\Gamma(1):\Gamma(\Ia_0)]^{1/2}
    \end{equation*}
    for every ideal $\Ia_0$ with sufficiently large index $[\Gamma(1):\Gamma(\Ia_0)]$.
\end{corollary}
  \begin{proof}
     The first statement follows readily from Theorem \ref{thmH1Estimate} together with the estimate in Corollary \ref{corEstimate}.
     The second statement is obvious if $s=1$, in this case we may take 
     $\kappa(F,D_0) = (2\pi)^{-2d}|\disc_F|^{3/2} \Delta(D_0)$. Note, that for $s=1$ the result holds for all $\Ia_0$.
     If $s=0$, then we have to take
     the index $[\Gamma(1):\Gamma(\Ia_0)]$ so large that $(2\pi)^{-2d}|\disc_F|^{3/2} \Delta(D_0)> [\Gamma(1):\Gamma(\Ia_0)]^{-1/2}$.
  \end{proof}

\subsection{Towards arbitrary groups}
From the previous Corollary we readily deduce the following weaker result, which in turn will imply the main theorem.
\begin{corollary}\label{corDecSeq}
  There is a decreasing sequence $\Gamma_1 \supset \Gamma_2 \supset \Gamma_3 \supset \dots$ of normal torsion-free subgroups of finite index in $\Gamma(1)$ 
  and a positive real number $\kappa > 0$ such that $\bigcap_i \Gamma_i = \{1\}$ and
  \begin{equation*}
      \dim H^1(\Gamma_i, \bbC) \geq \kappa [\Gamma(1):\Gamma_i]^{1/2}
  \end{equation*}
  for all $i$.
\end{corollary}
\begin{proof}
    Take any decreasing sequence $\Ia_1 \supset \Ia_2 \supset \Ia_3  \supset \dots$ of ideals in $\cO_F$ satisfying the assumptions of Corollary 
   \ref{corEstimateH1Index} and $\bigcap_i \Ia_i = \{0\}$. Finally, define $\Gamma_i = \Gamma(\Ia_i)$.
\end{proof}

\begin{maintheorem} Let $F$ be a totally real algebraic number field and let $E$ be
   a quadratic extension field having precisely one complex place. Let $D$ be a quaternion division algebra over $E$ which is 
   ramified in all real places of $E$. 
   Assume that $D$ is of the form $D \cong D_0 \otimes_F E$ for some quaternion algebra $D_0$ over $F$.
   
   Let $\Gamma \subseteq \SL_1(D)$ be an arithmetic group. There is a positive real number $\kappa > 0$ and
   a decreasing nested sequence $\{\Gamma_i\}_{i=1}^\infty$ of torsion-free subgroups of finite index in $\Gamma$
   satisfying $\bigcap_i \Gamma_i = \{1\}$ such that
   \begin{equation*}
      \dim H^1(\Gamma_i, \bbC) \geq \kappa [\Gamma:\Gamma_i]^{1/2}
   \end{equation*}
  for all $i$. Further, for every $i$ the group $\Gamma_i$ is normal in $\Gamma_1$.
\end{maintheorem}
\begin{proof}
     According to Corollary \ref{corDecSeq} there is a real number
     $\kappa' > 0$ and a decreasing sequence $\Gamma'_1 \supset \Gamma'_2 \supset\dots$ of torsion-free, finite index subgroups in $\Gamma(1)$ satisfying the
     claimed properties with respect to $\Gamma(1)$.

     Define $\Gamma_i := \Gamma \cap \Gamma'_i$. These are finite index subgroups in $\Gamma$ due to the assumption that $\Gamma$ is arithmetic.
     Clearly the $\Gamma_i$ intersect trivially. Since $\Gamma_i$ also has finite index in $\Gamma'_i$,
     we see $\dim H^1(\Gamma_i, \bbC) \geq \dim H^1(\Gamma'_i, \bbC)$. Define $\ell:= [\Gamma:\Gamma\cap\Gamma(1)]$.
     Further, the index satisfies 
    \begin{equation*}
        [\Gamma: \Gamma_i] = [\Gamma:\Gamma\cap\Gamma(1)][\Gamma\cap\Gamma(1):\Gamma_i] \leq \ell [\Gamma(1): \Gamma'_i].
    \end{equation*}
    Finally, we conclude
    \begin{equation*}
         \dim H^1(\Gamma_i, \bbC) \geq \dim H^1(\Gamma'_i, \bbC) \geq \kappa' [\Gamma(1): \Gamma'_i]^{1/2} \geq \kappa'\ell^{-1/2}[\Gamma: \Gamma_i]^{1/2}.
    \end{equation*}
     Since, $\Gamma'_i$ is normal in $\Gamma'_1$ for all $i$, we see that $\Gamma_i$ is normal in $\Gamma_1$. However, 
     the groups $\Gamma_i$ need not be normal in $\Gamma$.
\end{proof}

\section{The case of Bianchi groups}\label{sectBianchi}

\subsection{} In this section we make some comments on the classical case of Bianchi groups, which are
  \emph{non-cocompact} arithmetically defined subgroups of $\SL_2(\bbC)$.
  Let $F= \bbQ$ be the field of rational numbers and let $E$ be an imaginary quadratic number field.
  Moreover, let $\Ia \subseteq \cO_E$ be a non-trivial ideal 
  and define the principal congruence subgroup $\Gamma(\Ia) := \ker( \SL_2(\cO_E) \to \SL_2(\cO_E/\Ia))$ of level $\Ia$.
  We also use the notation $\Gamma(1) := \SL_2(\cO_E)$.

\subsection{}
 It is easy to obtain a result for Bianchi groups which is similar to the main theorem but with higher order of growth.
 Note that 
   \begin{equation*}
   [\Gamma(1):\Gamma(\Ia)] = |\SL_2(\cO_E/\Ia)| = \Nr(\Ia)^3 \prod_{\Ip | \Ia} (1-\Nr(\Ip)^{-2}),
  \end{equation*}
  where $\Nr(\Ia) = |\cO_E/\Ia|$. Assume that $\Gamma(\Ia)$ is torsion-free and let $h_\Ia$ denote the number of cusps of $\Gamma(\Ia)$.
  One can show that this number is given by
   \begin{equation*}
    h_\Ia = h_E |\mu_E|^{-1} \Nr(\Ia)^{-1} [\Gamma(1):\Gamma(\Ia)],
  \end{equation*}
  where $h_E$ is the ideal class number of $E$ and $\mu_E$ the (finite) group of units of $\cO_E$.
  The group $\Gamma(\Ia)$ acts freely and properly on hyperbolic three space
    \begin{equation*}
     \Hyp \cong \bsl{\SU(2)}{\SL_2(\bbC)}
    \end{equation*}
  and we obtain a non-compact hyperbolic manifold $\Hyp/\Gamma(\Ia)$.
  It follows from reduction theory that there is a compact manifold with boundary $M \subset \Hyp/\Gamma(\Ia)$
  such that the embedding $M \to \Hyp/\Gamma(\Ia)$ is a homotopy equivalence (cf. \cite[17.10]{Borel1969}).
  The boundary $\partial M$ of $M$ is a topologically disjoint union of $h_\Ia$ two-dimensional tori. 
  A general topological argument implies that the image of the restriction map
  \begin{equation*}
     r^1: H^1(M, \bbC) \longrightarrow H^1(\partial M, \bbC)
  \end{equation*}
  is a maximal isotropic subspace of $H^1(\partial M, \bbC)$
  with respect to the non-degenerate Poincar\'e pairing (see Lemme 11 in \cite{Serre1970} or 
  use the argument of the proof of VIII, 9.6 in \cite{Dold1980}).
  We conclude that 
    \begin{equation*}
       \dim H^1(\Gamma(\Ia),\bbC) \geq \dim(\im(r^1)) = \frac{1}{2} \dim H^1(\partial M, \bbC) = h_\Ia.
    \end{equation*}
  Summing up, it is easy to prove that
   \begin{equation}\label{eqEstimateBianchi}
       \dim H^1(\Gamma(\Ia),\bbC) \geq h_E |\mu_E|^{-1} \zeta_E(2)^{-1/3} [\Gamma(1):\Gamma(\Ia)]^{2/3}.
   \end{equation}
  Using the argument in the proof of the main theorem, one can obtain a similar result for arbitrary 
  arithmetic groups in $\SL_2(E)$.
  \begin{theorem}
     Let $E$ be an imaginary quadratic number field and let $\Gamma \subset \SL_2(E)$ be an arithmetic group. 
     There are a positive real number $\kappa > 0$ and
     a decreasing sequence $\{\Gamma_i\}_{i=1}^\infty$ (with trivial intersection) of torsion-free, finite index subgroups in $\Gamma$ 
     such that
             \begin{equation*}
                  \dim H^1(\Gamma_i, \bbC) \geq \kappa [\Gamma:\Gamma_i]^{2/3}  
             \end{equation*}
    for all $i \geq 1$. Moreover, the group $\Gamma_i$ is normal in $\Gamma_1$ for every index $i$.  
  \end{theorem}
  \begin{remark}
     Using the upper bounds of Calegari and Emerton \cite{CalegariEmerton2009} it follows that this is (in some cases)
     the correct asymptotic order of magnitude. Let $p$ be a prime number which splits in $E$ and let $\Ip \subseteq \cO_E$
     be a prime ideal of $\cO_E$ dividing $p$. In this case
     Theorem 3.4 of Calegari and Emerton \cite{CalegariEmerton2009} yields
     \begin{equation*}
         \dim H^1(\Gamma(\Ip^k), \bbC) = O(p^{2k}) 
     \end{equation*}
     as $k$ tends to infinity. As we have seen $[\Gamma(1):\Gamma(\Ip^k)] = p^{3k} (1-p^{-2})$,
     and together with \eqref{eqEstimateBianchi} it follows that
     \begin{equation*}
         \dim H^1(\Gamma(\Ip^k), \bbC) \:\asymp\: [\Gamma(1):\Gamma(\Ip^k)]^{2/3},
     \end{equation*}
     that is, both terms have the same order of magnitude as $k$ goes to infinity.
  \end{remark}
\subsection{The Lefschetz number}
Recall that the Lefschetz number formula obtained in Theorem \ref{thmLefschetzNumber} was independent of the assumptions made 
later on in Section \ref{sectHyperbolicApplications}. In particular, we may use it for Bianchi groups.

Let $d$ be a squarefree integer and let $E := \bbQ(\sqrt{d})$. Notice that we even do not assume
 that $d$ is negative in this paragraph. However, we assume that the extension $E/\bbQ$ is unramified over $2$, this
is the case precisely if
$d \equiv 1 \mod 4$. Let $m \geq 3$ be an integer and define the ideal $\Ia = m\cO_E$.
There is one split real place of $D_0 = M_2(\bbQ)$, i.e. $s=1$. Moreover, there are no real ramified places of $D_0$, hence $c=0$.
Finally, we see that $\rho(m) = |\{\:p \text{ prime number }\:|\: p|d \text{ and } \: p\nmid m \:\}|$. We define the congruence
 subgroup $\Gamma(m) :=  \Gamma(\Ia)$ in $\SL_2(\cO_E)$. We obtain the following Corollary to Theorem \ref{thmLefschetzNumber}. 
\begin{corollary}
  Let $E=\bbQ(\sqrt{d})$ be a quadratic number field for some squarefree integer $d \equiv 1 \mod 4$.
  Let $\sigma$ be the non-trivial Galois automorphism of $E/\bbQ$ and let $m \geq 3$ be an integer.
  Then 
   \begin{equation*}
     \calL(\sigma, \Gamma(m)) = - \frac{2^{\rho(m)} m^3}{12}\prod_{p|m}(1-p^{-2})
   \end{equation*}
  is the Lefschetz number of $\sigma$ in the cohomology of the principal congruence subgroup $\Gamma(m) \subset \SL_2(\cO_E)$.
\end{corollary}
\begin{proof}
    This follows from Theorem \ref{thmLefschetzNumber} using $\zeta(2) = \pi^2/6$.
\end{proof}
The Lefschetz number of the Galois automorphism acting on the full group $\PSL_2(\cO_E)$ has been calculated by Rohlfs \cite{Rohlfs1985}.
A formula for the Lefschetz number of $\sigma$ on congruence subgroups in $\SL_2(\cO_E)$ (without restrictions on $d$)
has recently been announced by Seng\"un and T\"urkelli.

\section{Appendix: Local calculations}\label{secAppendixLocal}

\subsection{} In this appendix we gather those results for the non-abelian Galois cohomology sets $H^1$ which
can be stated locally. In this section $F$ denotes a finite extension of some $p$-adic field $\bbQ_p$ where $p$ is a prime
number. We write $\co_0$ for the valuation ring of $F$ and we choose a uniformizer $\pi_0 \in \co_0$ which generates 
the prime ideal $(\pi_0)=\Ip_0 \subset \co_0$.
The residue class field $\co_0/\Ip_0$ will be denoted $k_0$.
Moreover, let $E/F$ be a quadratic extension. The valuation ring of $E$ will be denoted by $\co$, and let $\pi$ be a generator
of the prime ideal $\pi\co = \Ip \subset \co$. The residue field of $E$ is denoted $k$.
The non-trivial Galois automorphism of $E/F$ will be refered to as $\sigma$.

\subsection{} Let $D_0$ be a quaternion algebra defined over $F$ and let $\Lambda_0$ denote a maximal $\co_0$-order in $D_0$.
We get the quaternion algebra $D := D_0 \otimes_F E$ over $E$ with the order $\Lambda = \Lambda_0 \otimes_{\co_0} \co$. 
It is important to understand that this order need not be a maximal $\co$-order of $D$.
One should further notice that $D$ is always isomorphic to the matrix algebra $M_2(E)$, since every quadratic extension
splits $D_0$ (cf. Thm.1.3 in \cite[p.33]{Vigneras1980}).
Moreover, we define the group schemes $G_0$ and $G$ over $\co_0$ just as in \ref{subsecGroupSchemes}.

For every integer $j \geq 1$ we define the open compact subgroup $K(j)$ as the kernel
 of the reduction map $G(\co_0) \to G(\co_0/\Ip_0^j)$.
 Further, we set $K(0) := G(\co_0)$. These subgroups are $\sigma$-stable, and we want to understand the cohomology sets $H^1(\sigma,K(j))$.

\subsection{Basic observation}\label{sectBasicObservation} We want to determine the first non-abelian cohomology $H^1(\sigma, G(\co_0))$.
  In order to do this, we mimic the proof
  of (29.2) in \cite{BOInv}, but we work with rings instead of fields. 
  Let $b \in Z^1(\sigma, \GL_1(\Lambda))$ be a cocycle. 
   We define the fixed point space
   \begin{equation*}
       U(b):= \{\:x \in \Lambda \:|\: b \up{\sigma}x = x\:\},
   \end{equation*}
    which clearly is a right $\Lambda_0$-module. It follows from the theory of Galois descent that
     the canonical map 
  \begin{equation*}
    \phi_b: U(b) \otimes_{\co_0} \co  \to \Lambda
  \end{equation*}
   is injective and that the image is an $\co$-lattice of finite index in $\Lambda$.
   The $\co_0$-module $U(b)$ is free and we deduce that $U(b)$ is of $\co_0$-rank four.
   As $\Lambda_0$ is a right
   principal ideal ring (see (17.3) in \cite{Reiner2003}), we see that $U(b)$ is isomorphic to $\Lambda_0$ as right
   $\Lambda_0$-module.
   We choose a generator $g \in U(b)$, i.e. every $x \in U(b)$ can be written $x = gy$ for some $y \in \Lambda_0$.

   Observe that, given two equivalent cocycles $b, b'$ with $c \in \GL_1(\Lambda)$ 
   satisfying $b' = c^{-1} b \up{\sigma}c$,
   it follows that $U(b) = c U(b')$ and similarly $\im(\phi_b) = c \im(\phi_{b'})$. This means if such a relation is not possible, we can use 
   the images of $\phi_b$ and $\phi_{b'}$ to exclude that $b$ and $b'$ are equivalent.
   This setting will be used in the proofs of the following results.

\subsection{The unramified case}
In this section we will assume that the extension $E/F$ is unramified.

Suppose $D_0$ is a matrix algebra, 
then the order $\Lambda = \Lambda_0 \otimes \co$ is maximal and isomorphic to
the full matrix algebra $M_2(\co)$ (cf. \cite[(17.3)]{Reiner2003}). In particular, the reduced norm $\nrd: \Lambda \to \co$ is onto.

On the other hand, if $D_0$ is the unique quaternion division algebra over $F$, then
$\Lambda = \Lambda_0 \otimes \co$ is not maximal.
More precisely, there is an isomorphism of $E$ algebras
\begin{equation*}
    D \stackrel{\simeq}{\longrightarrow} M_2(E) 
\end{equation*}
which maps the order $\Lambda$ to $\{\:\bigl(\begin{smallmatrix}x&y\\ \pi z & w\end{smallmatrix}\bigr)\:|\: x, y, z, w \in \co\:\}$.
This means $\Lambda$ is an Eichler order of level $\pi\co$.
However, the reduced norm $\nrd: \Lambda \to \co$ is again surjective.

\begin{lemma} \label{lemmaLocalH1Unram}
   Suppose $E/F$ is unramified. In this case $H^1(\sigma, \GL_1(\Lambda)) = \{1\}$.
\end{lemma}
\begin{proof}
   First, choose $\pi = \pi_0$.
   We want to show that $\phi_b$ is surjective.
   We find an element $\zeta \in \co$ such that $E = F(\zeta)$ and $\co = \co_0 \oplus \zeta \co_0$.
   Note that $\up{\sigma}\zeta-\zeta \not\equiv 0 \mod \pi_0$ since $\zeta + \pi_0\co \notin k_0$. Consequently,
   $\up{\sigma}\zeta  - \zeta$ is a unit in $\co$ and we choose $u := (\up{\sigma}\zeta  - \zeta)^{-1}$.
   Take any $v \in \Lambda$, then $v_1 = v + b\up{\sigma}v$ and $v_2 = \zeta v + b\up{\sigma}{\zeta}\up{\sigma}{v}$ are in $U(b)$.
   Finally, we conclude that $v = \up{\sigma}\zeta u v_1 - u v_2 \in \im(\phi)$.
   This means every element in $\Lambda$ can be written as $gy$ for some $y \in \Lambda$. We deduce that $g$ is a unit in $\Lambda$
   and $b = g \up{\sigma}g^{-1}$ since $g \in U(b)$.
\end{proof}

\begin{corollary}\label{corLocalUnramified1}
  If the extension $E/F$ is unramified, then $H^1(\sigma, G(\co_0)) = \{1\}$.
\end{corollary}
  \begin{proof}
    Recall that $G(\co_0) = \SL_1(\Lambda)$ and that the reduced norm $\nrd: \Lambda \to \co$ is surjective.
    Hence, there is a short exact sequence of groups with $\sigma$-action
    \begin{equation*}
       1 \longrightarrow \SL_1(\Lambda) \longrightarrow \GL_1(\Lambda) \stackrel{\nrd}{\longrightarrow} \co^{\times} \longrightarrow 1.
    \end{equation*}
   In turn there is an induced long exact sequence of pointed sets
   \begin{equation*}
      1 \longrightarrow \SL_1(\Lambda_0) \longrightarrow \GL_1(\Lambda_0) \stackrel{\nrd}{\longrightarrow} \co_0^{\times}
       \longrightarrow H^1(\sigma, \SL_1(\Lambda)) \longrightarrow 1.
   \end{equation*}
   Again, the reduced norm $\Lambda_0^\times \to \co_0^\times$ is surjective. This is clear, if $D_0$ is a matrix algebra. If $D_0$ is the 
   unique central division algebra of dimension four over $F$, then this follows from
   the fact that $E$ is embedded in $D_0$ as a maximal subfield and so $\nrd(\Lambda_0^\times) \supseteq \N_{E/F}(\co^\times) = \co_0^\times$.
  \end{proof}

\begin{lemma}\label{lemLocalUnramified2}
 Assume that the extension $E/F$ is unramified. In this case 
\begin{equation*}
  H^1(\sigma, K(j)) =\{1\}
\end{equation*}
for every integer $j \geq 0$.
\end{lemma}
\begin{proof}
  The statement for $j=0$ was proven in Corollary \ref{corLocalUnramified1}. Let $j\geq 1$, the short sequence of groups
  \begin{equation*}
      1 \longrightarrow K(j) \longrightarrow G(\co_0) \longrightarrow G(\co_0/\Ip_0^j) \longrightarrow 1
  \end{equation*}
   is exact, since the group scheme $G$ is smooth over $\co_0$.
   Consider the induced long exact sequence of pointed sets
  \begin{equation*}
      G_0(\co_0) \stackrel{f}{\longrightarrow} G_0(\co_0/\Ip_0^j) \longrightarrow H^1(\sigma, K(j)) \longrightarrow 1.
  \end{equation*}
   Note that the group scheme of fixed points $G^\sigma$ is isomorphic to $G_0$ over $\co_0$ since $E/F$ is unramified. 
   The reduction map $f$ is surjective and the claim follows.
\end{proof}

\subsection{The ramified case}
We assume that $E/F$ is a ramified extension. Here the situation becomes quite tedious. 
For the sake of simplicity we will assume later on that $p \neq 2$.

As before, if $D_0$ is a matrix algebra, then $\Lambda = \Lambda_0 \otimes \co$ is a maximal order and isomorphic to the full matrix algebra
$M_2(\co)$.

Assume now that $D_0$ is the unique central division algebra of dimension $4$ over $F$. Let $W/F$ be the unramified quadratic
extension of $F$ and let $\co_W$ be its valuation ring. The unramified extension
$W$ of $F$ is embedded into $D_0$ as a maximal subfield such that $D_0 = W \oplus W \omega$ with $\omega^2 = \pi_0$.
The maximal order $\Lambda_0$ is $\Lambda_0 = \co_W \oplus \co_W \omega$ with respect to this 
decomposition (cf. Vign\'eras \cite[Cor. 1.7. p.34]{Vigneras1980}).
We define $L := W \otimes_F E$, this is a field extension of degree $4$ over $F$. Further, the extension $L/E$ is unramified,
whereas $L/W$ is a ramified extension. Let $\co_L$ be the valuation ring of $L$, we have $\co_L \cong \co_W \otimes \co$.
Consequently, the order $\Lambda = \Lambda_0 \otimes \co$ is isomorphic to $\co_L \oplus \co_L \omega$ with the 
appropriate multiplication.
One can check that there is precisely one proper right ideal $I \subset \Lambda$ strictly containing $\pi \Lambda$, namely
$I = \pi \co_L \oplus \co_L \omega$. Moreover, one can verify by calculation that this right ideal
can not be generated by one element.  

 \begin{lemma}\label{lemLocalRamified1}
    Assume $p \neq 2$. If $E/F$ is a ramified extension, then
  \begin{equation*}
       H^1(\sigma, \SL_1(\Lambda)) = \{\pm 1\}.
  \end{equation*}
 \end{lemma}
  \begin{proof}
   Return to the setting of \ref{sectBasicObservation}. We proceed in a similar fashion as in the proof of Lemma \ref{lemmaLocalH1Unram} but we
   assume directly that $b \in Z^1(\sigma, \SL_1(\Lambda))$.
   Using the assumption that $p$ is odd, we
   may further assume $\pi^2 = u\pi_0$ for some unit $u \in \co^\times_0$. 
   Note that $\co = \co_0 \oplus \pi \co_0$ and $\up{\sigma}\pi = -\pi$.
   
   Take an arbitrary $v \in \Lambda$, we claim that $\pi v \in \im(\phi_b)$. The two elements $v_1 = v + b \up{\sigma}v$ and 
   $v_2 = \pi v - b \pi \up{\sigma}v$ are in $U(b)$. Clearly, $2 \pi v = \pi v_1 + v_2$ and the claim follows, since $2$ is a unit in $\co_0$. 
   Consequently, we have $\pi \Lambda \subseteq \im(\phi_b) \subseteq \Lambda$.
   The image of $\phi_b$ is a right ideal in $\Lambda$.
   We distinguish three cases:

   Case 1: Suppose $\im(\phi_b) = \Lambda$, then the generator $g \in U(b)$ is a unit in $\Lambda$ and satisfies $b = g \up{\sigma}g^{-1}$.
   From $\nrd(b) = 1$ we deduce that $\nrd(g) = \up{\sigma}\nrd(g)  \in \co_0^\times$. Multiplying $g$ from the right with an element in $\Lambda_0^\times$
   having reduced norm $\nrd(g)^{-1}$, we see that $b$ represents the trivial class in $H^1(\sigma,\SL_1(\Lambda))$.

   Case 2: Suppose $\im(\phi_b) = \pi\Lambda$. The generator $g \in U(b)$ is of the form $\pi h$, where $h \in \Lambda^\times$. From 
    this we see the relation $b = -h\up{\sigma} h^{-1}$. As in case one we can achieve that $h$ has reduced norm $1$ and so 
    $b$ represents the class of $-1$ in $H^1(\sigma, \SL_1(\Lambda))$. By the way, using the last remark made in \ref{sectBasicObservation}, 
    it follows that the cocylces $1$ and $-1$ can not be equivalent (even over $\GL_1(\Lambda)$).

   Case 3: Suppose $\pi\Lambda \subsetneq \im(\phi_b) \subsetneq \Lambda$. We distinguish whether $D_0$ is split or not.

   Suppose $D_0 \cong M_2(F)$ and choose an isomorphism $\Lambda \cong M_2(\co)$. In this case
   we know that $\im(\phi_b)$ is generated (as right ideal) by an element of the form $a \delta$ where $a \in \GL_2(\co)$ and 
   \begin{equation*}
       \delta = \begin{pmatrix}
                   1 & 0 \\
                   0 & \pi \\
                \end{pmatrix}
   \end{equation*}
  (cf. (17.7) \cite{Reiner2003}). It follows that the generator $g \in U(b)$ is $g = a\delta c$ for some unit $c \in \Lambda^\times$.
   We get $b \up{\sigma}(a\delta c)= a \delta c$. Applying the reduced norm, we find $\up{\sigma}(\nrd(a)\nrd(c)) = - \nrd(a)\nrd(c)$.
   This implies $\nrd(ac) \in \pi\co$ which is a contradiction to $a$ and $c$ being units.

  Suppose that $D_0$ is a division algebra. Since $U(b)$ is generated by one element, the same must be true for $\im(\phi_b)$.
  However, as pointed out before, there is no such right ideal in $\Lambda$ properly containing $\pi\Lambda$.
  \end{proof}

\begin{lemma}\label{lemLocalRamified2}
 Let $p \neq 2$ and let $E/F$ be a ramified extension. For every $j \geq 1$ the first cohomology $H^1(\sigma,K(j))$ is trivial. 
\end{lemma}
\begin{proof}
  We claim that the map $H^1(\sigma, K(1)) \to H^1(\sigma,G(\co_0))$ is trivial.
  To see this, suppose that $-1$ is equivalent to a cocyle $b \in K(1)$. Under this assumption there is some $c \in \SL_1(\Lambda)$ such that
  $-1 = c^{-1}b\up{\sigma}c$. Considering this equation modulo $\pi_0$, we get $-c \equiv \up{\sigma}c \mod \pi_0$. Let $\pi \in \co$ denote, as before, a
  uniformizer satisfying $\pi^2 = u\pi_0$, $\up{\sigma}\pi = -\pi$ and $\co = \co_0 \oplus \pi \co_0$, we deduce $c \in \pi \Lambda$. This is
  a contradiction to the assumption that $c$ is a unit, which proves the claim.

  Finally, apply the argument of Lemma \ref{lemLocalUnramified2} using that $G^\sigma = G_0$ since $p\neq 2$.
\end{proof}

\begin{remark}\label{remRohlfsGeneral}
  Many results of this section can be deduced from Rohlfs general treatment (see Satz 2.6 and Korollar 2.7 in \cite{Rohlfs1978}).
  Since most results follow directly in the given situation we decided to provide independent proofs.

  It seems to be a more difficult task to give a general description of $H^1(\sigma, G(\co_0))$ in the ramified case
  when the residual characteristic is $p=2$.
  One can not expect a simple result like Lemma \ref{lemLocalRamified2}. 
  This follows from the work of Rohlfs, who determined the cohomology sets for quadratic extensions of $\bbQ$ 
  (cf. Table to Satz 4.1 in \cite{Rohlfs1978}).
  For the applications we have in mind it is sufficient to know that the cohomology set $H^1(\sigma, G(\co_0))$ is finite (see Kor. 2.5 in \cite{Rohlfs1978}).
\end{remark}

\subsection{The orders of certain finite groups}
In this section we summarize some results on the cardinalities of the involved finite groups.
These results are well-known or can be obtained using the well-known tricks. We simply gather these results here.
We keep the notation used throughout the appendix. 
In particular, $F$ denotes a finite extension of some $p$-adic field $\bbQ_p$ and $E$ is a quadratic extension field of $F$. 
We write $\Nr(\Ip_0)$ for the cardinality of the residue class field $k_0 = \co_0/\Ip_0$.
\begin{lemma}\label{lemCounting1}
For every positive integer $e$ the following holds:
\begin{equation*}
 |\SL_2(\co_0/\Ip_0^e)| = \Nr(\Ip_0)^{3e}(1-\Nr(\Ip_0)^{-2})
\end{equation*}
\end{lemma}
\begin{lemma}\label{lemCounting2}
Let $e$ be a positive integer and assume $D_0$ is a division algebra, then 
$|G_0(\co_0/\Ip_0^e)| = \Nr(\Ip_0)^{3e}(1+\Nr(\Ip_0)^{-1})$. Moreover,
\begin{itemize}
 \item if $E/F$ is unramified,
                       then $|G(\co_0/\Ip_0^e)| = \Nr(\Ip_0)^{6e}(1-\Nr(\Ip_0)^{-2})$,
 \item if $E/F$ is ramified, then $|G(\co_0/\Ip_0^e)| = \Nr(\Ip_0)^{6e}(1+\Nr(\Ip_0)^{-1})$,
\end{itemize}
 
\end{lemma}
\begin{proof}
  We only indicate the proof for the claim when $E/F$ is unramified. In this case $\Lambda$ is an Eichler order of level $\pi\co$, i.e.
  \begin{equation*}
    \Lambda \cong \{\:\begin{pmatrix}x&y\\ \pi z & w\end{pmatrix}\:|\: x, y, z, w \in \co\:\}.
  \end{equation*}
  One counts the group of units $|(\Lambda/\pi \Lambda)^\times| = |k|^2 (|k|-1)^2 = \Nr(\Ip_0)^8(1-\Nr(\Ip_0)^{-2})^2$
  and (by the usual trick) one obtains  $|(\Lambda/\pi^e \Lambda)^\times| = \Nr(\Ip_0)^{8e}(1-\Nr(\Ip_0)^{-2})^2$.
  The reduced norm $\nrd: (\Lambda/\pi^e \Lambda)^\times \to (\co/\pi^e\co)^\times$ is onto and so
  \begin{equation*}
      |\SL_1(\Lambda/\pi^e\Lambda)| = \Nr(\Ip_0)^{6e}(1-\Nr(\Ip_0)^{-2}).
  \end{equation*}
\end{proof}
For a positive integer $e$, define 
\begin{equation*}
  Q_e := \frac{|G_0(\co_0/\Ip_0^e)|}{\sqrt{|G(\co_0/\Ip_0^e)|}}.
\end{equation*}
With the help of Lemma \ref{lemCounting1} and Lemma \ref{lemCounting2} is easy to verify to following assertions.
\begin{lemma}\label{lemLocalCalculationIndex}
 \begin{enumerate}
  \item If $E/F$ is unramified and $D_0$ is split, then 
  \begin{equation*}
     Q_e^2 = (1-\Nr(\Ip_0)^{-2})(1+\Nr(\Ip_0)^{-2})^{-1}.
  \end{equation*}
  \item If $E/F$ is unramified and $D_0$ is a division algebra, then 
  \begin{equation*}
     Q_e^2 = (1+\Nr(\Ip_0)^{-1})(1-\Nr(\Ip_0)^{-1})^{-1}.
  \end{equation*}
  \item If $E/F$ is ramified and $D_0$ is split, then 
  \begin{equation*}
     Q_e^2 = 1-\Nr(\Ip_0)^{-2}.
  \end{equation*}
   \item If $E/F$ is ramified and $D_0$ is a division algebra, then 
  \begin{equation*}
     Q_e^2 = 1+\Nr(\Ip_0)^{-1}.
  \end{equation*}
 \end{enumerate}
\end{lemma}

\newpage
\bibliography{/users/kionke/Documents/Dissertation/Library}{}
\bibliographystyle{amsplain}
\end{document}